\documentclass[english]{amsart}
\usepackage{graphicx}
\usepackage{amsmath,amsfonts,amsthm}
\usepackage{amssymb}
\usepackage{multirow,array}
\usepackage{comment}
\usepackage{enumitem}
\usepackage[utf8]{inputenc}
\usepackage[T1]{fontenc}
\usepackage{emptypage}
\usepackage{hyperref}
\usepackage{color}
\usepackage{mathrsfs}
\usepackage{marvosym}
\usepackage{babel}
\usepackage{url}
\usepackage{amsbsy}
\usepackage{shuffle}
\usepackage{tikz}
\usepackage{mhequ}

\usetikzlibrary{arrows,snakes,backgrounds}

\usetikzlibrary{snakes}
\usetikzlibrary{decorations}
\usetikzlibrary{positioning}
\usetikzlibrary{shapes}

\tikzset{
root/.style={circle,fill=black!50,inner sep=0pt, minimum size=3mm},
        dot/.style={circle,fill=black,inner sep=0pt, minimum size=1.5mm},
         bluedot/.style={circle,fill=blue,inner sep=0pt, minimum size=1.5mm},
        reddot/.style={circle,fill=red,inner sep=0pt, minimum size=1.5mm},
        var/.style={circle,fill=black!10,draw=black,inner sep=0pt, minimum size=3mm},
        kernel/.style={semithick,shorten >=2pt,shorten <=2pt},
        kernel1/.style={thick},
        kernels/.style={snake=zigzag,shorten >=2pt,shorten <=2pt,segment amplitude=1pt,segment length=4pt,line before snake=2pt,line after snake=5pt,},
		kernels1/.style={snake=zigzag,segment amplitude=0.5pt,segment length=2pt},
		rho1/.style={dotted,semithick},
        rho/.style={densely dashed,semithick,shorten >=2pt,shorten <=2pt},
           testfcn/.style={dotted,semithick,shorten >=2pt,shorten <=2pt},
        renorm/.style={shape=circle,fill=white,inner sep=1pt},
        labl/.style={shape=rectangle,fill=white,inner sep=1pt},
        xic/.style={very thin,circle,fill=symbols,draw=black,inner sep=0pt,minimum size=1.2mm},
        xi/.style={very thin,circle,fill=blue!10,draw=black,inner sep=0pt,minimum size=1.2mm},
        xix/.style={crosscircle,fill=blue!10,draw=black,inner sep=0pt,minimum size=1.2mm},
	xib/.style={very thin,circle,fill=blue!10,draw=black,inner sep=0pt,minimum size=1.6mm},
	xie/.style={very thin,circle,fill=green!50!black,draw=black,inner sep=0pt,minimum size=1mm},
	xid/.style={very thin,circle,fill=symbols,draw=black,inner sep=0pt,minimum size=1.6mm},
	xibx/.style={crosscircle,fill=blue!10,draw=black,inner sep=0pt,minimum size=1.6mm},
	edgetype/.style={very thin,circle,draw=black,inner sep=0pt,minimum size=5mm},
	nodetype/.style={very thick,circle,draw=black,inner sep=0pt,minimum size=5mm},
	kernels2/.style={very thick,draw=connection,segment length=12pt},
clean/.style={thin,circle,fill=black,inner sep=0pt,minimum size=1mm},	 not/.style={thin,circle,fill=symbols,draw=connection,fill=connection,inner sep=0pt,minimum size=0.5mm},
	>=stealth,
        }

\DeclareFontFamily{U}{matha}{\hyphenchar\font45}
\DeclareFontShape{U}{matha}{m}{n}{
      <5> <6> <7> <8> <9> <10> gen * matha
      <10.95> matha10 <12> <14.4> <17.28> <20.74> <24.88> matha12
      }{}
\DeclareSymbolFont{matha}{U}{matha}{m}{n}
\DeclareFontSubstitution{U}{matha}{m}{n}
\DeclareFontFamily{U}{mathx}{\hyphenchar\font45}
\DeclareFontShape{U}{mathx}{m}{n}{
      <5> <6> <7> <8> <9> <10>
      <10.95> <12> <14.4> <17.28> <20.74> <24.88>
      mathx10
      }{}
\DeclareSymbolFont{mathx}{U}{mathx}{m}{n}
\DeclareFontSubstitution{U}{mathx}{m}{n}
\DeclareMathDelimiter{\vvvert}{0}{matha}{"7E}{mathx}{"17}

\newcommand{\Wve}{{W_{\eps}}}
\newcommand{\BWve}{{\BW_{\eps}}}
\newcommand{\BBWve}{{\BBW_{\eps}}}

\newcommand{\Wvn}{{W_{n}}}
\newcommand{\BWvn}{{\BW_{n}}}
\newcommand{\BBWvn}{{\BBW_{n}}}

\newcommand{\floor}[1]{\lfloor #1 \rfloor}

\newcommand{\tY}{\tilde{Y}}
\newcommand{\tBX}{\tilde{\mathbf{X}}}

\newcommand{\MB}{\mathcal{B}}

\newcommand{\mfB}{\mathfrak{B}}

\newcommand{\ty}{\tilde{y}}

\newcommand{\BBW}{\mathbb{W}}

\newcommand{\BBB}{\mathbb{B}}

\newcommand{\CC}{{\pmb{\mathscr{C}}}}
\newcommand{\DD}{{\pmb{\mathscr{D}}}}

\newtheorem{thm}{Theorem}[section]
\newtheorem{assumption}[thm]{Assumption}

\newtheorem{rem}[thm]{Remark}
\newtheorem{prop}[thm]{Proposition}

\let\oldtocsection=\tocsection

\let\oldtocsubsection=\tocsubsection

\let\oldtocsubsubsection=\tocsubsubsection

\renewcommand{\tocsection}[2]{\hspace{0em}\oldtocsection{#1}{#2}}
\renewcommand{\tocsubsection}[2]{\hspace{1em}\oldtocsubsection{#1}{#2}}
\renewcommand{\tocsubsubsection}[2]{\hspace{2em}\oldtocsubsubsection{#1}{#2}}

\newcommand{\E}{\mathbb{E}}
\newcommand{\R}{\mathbb{R}}

\newcommand{\X}{\mathbb{X}}
\newcommand{\BX}{\mathbf{X}}

\newcommand{\BB}{\mathbf{B}}
\newcommand{\op}{\mathcal{P}}

\newcommand{\Z}{\mathbb{Z}}

\newcommand{\BW}{\mathbf{W}}

\newcommand{\supp}{\operatorname{supp}}

\newcommand{\var}{\textnormal{-var}}

\newcommand{\Sym}{\operatorname{Sym}}
\newcommand{\Hol}{\textnormal{-H{\"o}l}}
\newcommand{\eps}{\varepsilon}
\newcommand{\vertiii}[1]{{\left\vert\kern-0.25ex\left\vert\kern-0.25ex\left\vert #1
    \right\vert\kern-0.25ex\right\vert\kern-0.25ex\right\vert}}

\def\var{\text{-}\textrm{var}}
\def\Hol{\text{-}\textrm{H\"ol}}

 \newcommand{\dist}{\operatorname{dist}}
\newcommand{\Leb}{\operatorname{Leb}}
\newcommand{\Vol}{\operatorname{Vol}}

\newcommand{\cB}{{\mathcal{B}}}

\begin{document}

\title{Multiscale systems, homogenization, and rough paths}

\author[I. Chevyrev]{Ilya Chevyrev}
   % Address of record for the research reported here
\address{I. Chevyrev,
Mathematical Institute,
University of Oxford,
Andrew Wiles Building,
Radcliffe Observatory Quarter,
Woodstock Road,
Oxford OX2 6GG,
United Kingdom}
\email{chevyrev@maths.ox.ac.uk}

\author[P.K. Friz]{Peter K. Friz}
\address{P.K. Friz, Institut f\"ur Mathematik, Technische Universit\"at Berlin, and Weierstra\ss --Institut f\"ur Angewandte Analysis und Stochastik, Berlin, Germany}
\email{friz@math.tu-berlin.de}

\author[A. Korepanov]{Alexey Korepanov}
\address{A. Korepanov,
Mathematics Institute,
University of Warwick,
Coventry, CV4 7AL,
United Kingdom}
\email{a.korepanov@warwick.ac.uk}

\author[I. Melbourne]{Ian Melbourne}
\address{I. Melbourne,
Mathematics Institute,
University of Warwick,
Coventry, CV4 7AL,
United Kingdom}
\email{i.melbourne@warwick.ac.uk}

\author[H. Zhang]{Huilin Zhang}
   % Address of record for the research reported here
\address{H. Zhang, Institute of Mathematics, Fudan University, Shanghai, 200433, China}
\email{huilinzhang2014@gmail.com}

\dedicatory{Dedicated to Professor S.R.S Varadhan on the occasion of his 75th birthday}

%\subjclass[2010]{Primary 60H99; Secondary 60H10}

\keywords{Fast-slow systems, homogenization, rough paths}

\begin{abstract}
In recent years, substantial progress was made towards understanding convergence of fast-slow deterministic systems
to stochastic differential equations. In contrast to more classical approaches, the assumptions on the fast flow are very mild.
We survey the origins of this theory and then revisit and improve the analysis of Kelly-Melbourne [Ann.\ Probab.\ Volume 44, Number 1 (2016), 479-520],
taking into account recent progress in $p$-variation and c{\`a}dl{\`a}g rough path theory.
\end{abstract}

\maketitle

\tableofcontents

%\setcounter{section}{-1}
%%%%%%%%%%%%%%%%%%%%%%%%%%%%%%%%%%%%%%%%%%%%%%%%%%%%%%%%%%%%%%%%%%%%%%%%%%%%%%%%%%%%%%%%%%%%%%%%%%%%%%%%%%%%%%%%%%%%%%%%
\section{Introduction}
\label{sec:intro}

The purpose of this article is to survey and improve several recent developments in the theories of homogenization and rough paths, and the interaction between them.
From the side of homogenization, we are interested in the programme initiated by~\cite{MS11} and continued in~\cite{GM13} of studying fast-slow systems without mixing assumptions on the fast flow.
From the side of rough paths, we are interested in surveying recent extensions of the theory to the discontinuous setting~\cite{FS17, CF17x, FZ17} (see also~\cite{C15,Davie07,Kelly16,Williams01} for related results); the continuous theory, for the purposes of this survey,  is well-understood~\cite{FV10}.
The connection between the two sides first arose in~\cite{KM16,KM17} in which the authors were able to employ rough path techniques (in the continuous and discontinuous setting) to study systems widely generalising those considered in~\cite{MS11,GM13}.

In this article we address both continuous and discrete systems.
The continuous fast-slow systems take the form of the ODEs
\begin{align} \label{eq:ODE}
\frac{d}{dt} x_\eps &=a(x_\eps,y_\eps) +\eps^{-1}b(x_\eps,y_\eps)\;, \qquad
\frac{d}{dt} y_\eps =\eps^{-2}g(y_\eps)\;.
\end{align}
The equations are posed on $\R^d\times M$ for some compact Riemannian manifold $M$,
and $g:M\to TM$ is a suitable vector field.
We assume a fixed initial condition $x_\eps(0)=\xi$ for some fixed
$\xi\in\R^d$, while the initial condition for $y_\eps$ is drawn randomly from $(M,\lambda)$, where $\lambda$ is a Borel probability measure on $M$.

For the discrete systems, we are interested in dynamics of the form
\begin{equation}
\label{equ:discreteFastSlowGeneral}
X^{(n)}_{j+1} = X^{(n)}_j + n^{-1}a(X_j^{(n)},Y_j) + n^{-1/2}b(X_j^{(n)},Y_j)\;,
\quad
Y_{j+1} = TY_j\;.
\end{equation}
The equations are again posed on $\R^d\times M$, and $T:M\to M$ is an appropriate transformation.
As before, $X^{(n)}_0 = \xi \in \R^d$ is fixed and $Y_0$ is drawn randomly from a probability measure $\lambda$ on $M$.

Let $x_\eps : [0,1] \to \R^d$ denote either the solution to~\eqref{eq:ODE}, or the piecewise constant path $x_{\eps}(t)= X^{\floor{1/\eps^2}}_{\floor{t/\eps^2} }$, where $X_j^{(n)}$ is the solution to~\eqref{equ:discreteFastSlowGeneral}.
The primary goal of this article is to show convergence in law $x_\eps \to X$ in the uniform (or stronger) topology as $\eps \to 0$.
Here $X$ is a stochastic process, which in our situation will be the solution to an SDE.

Throughout this note we shall focus on the case where $a(x,y) \equiv a(x)$ depends only on $x$ and $b(x,y) \equiv b(x)v(y)$, where $v:M\to \R^m$ is an observable of $y$ and $b : \R^d \to L(\R^m,\R^d)$ This is precisely the situation considered in~\cite{KM16}. One restriction of the method in~\cite{KM16} is the use of H{\"o}lder rough path topology which necessitates moment conditions on the fast dynamics which are suboptimal from the point of view of homogenization.
Our main insight is that switching from $\alpha$-H{\"o}lder to $p$-variation rough path topology allows for optimal moment assumptions on the fast dynamics. The non-product case was previously handled, also with suboptimal moment assumptions, in \cite{KM17} and also \cite{BC17}, using infinite-dimensional and flow-based rough paths respectively. We briefly discuss this and some other extensions in Section~\ref{sec:exts}, leaving  
a full analysis of the general (non-product) case, under equally optimal moment assumptions, to a forthcoming artice \cite{CFKMZPrep}.

\medskip

{\bf An example of fast dynamics.} There are many examples to which the results presented here apply, however we feel it is important to have a concrete (and simple to state) example in mind from the very beginning. In this regard, Pomeau \& Manneville~\cite{PomeauManneville80} introduced a class of maps that exhibit {\em intermittency} as part of their study of turbulent bursts.
The most-studied example~\cite{LSV99} is the one-dimensional map $T:M\to M$, $M=[0,1]$,
given by
\begin{align} \label{eq:LSV}
Ty=\begin{cases} y(1+2^\gamma y^\gamma) & y<\frac12 \\ 2y-1 & y\ge\frac12
\end{cases}.
\end{align}
Here $\gamma\ge0$ is a parameter.  When $\gamma=0$ this is the doubling map $Ty=2y\bmod1$ which is uniformly expanding (see Section~\ref{sec:chaos}).
For $\gamma>0$, there is a neutral fixed point at $0$ ($T'(0)=1$) which has more and more influence as $\gamma$ increases.
For each value of $\gamma\in[0,1)$, there is a unique absolutely continuous invariant probability measure $\mu$.  This measure is ergodic and equivalent (in fact equal when $\gamma=0$) to the Lebesgue measure.

Suppose that $v \colon M \to \R^m$ is H\"older continuous and $\int v \, d\mu = 0$.
Let
\[
  v_n=\sum_{0\le j<n}v\circ T^j.
\]
By~\cite{LSV99,Young99}, for $\gamma \in [0,\frac{1}{2})$, the random variable
$n^{-1/2} v_n$, defined on the probability space $(M, \mu)$, converges in law to a normal distribution.
(Convergence in law also holds on $(M, \Leb)$.)
In other words, the \emph{central limit theorem} (CLT) holds.
However, by~\cite{Gouezel04}, the CLT fails for $\gamma>\frac12$ (instead there is convergence to a stable law of index $\gamma^{-1}$); for $\gamma=\frac12$ the CLT holds but with non-standard normalization $(n\log n)^{-1/2}$.
Hence from now on we restrict to $\gamma \in [0, \frac{1}{2})$.

Define
\[
S_n=\sum_{0\le i\le j<n}(v\circ T^i)\otimes(v\circ T^j)\;.
\]
The approach to homogenization of fast-slow systems in~\cite{KM16}
requires convergence of the pair of stochastic processes
$\bigl(n^{-1/2} v_{\lfloor n t \rfloor}, n^{-1} S_{\lfloor n t\rfloor}\bigr)$ to
an enhanced Brownian motion, which is established for all $\gamma \in [0,\frac12)$.
(See Sections~\ref{sec:chaos} and~\ref{sec:chaos2}.)
Further, the approach based on H\"older rough path theory requires that
$\|v_n\|_{2q}=O(n^{1/2})$ and
$\|S_n\|_{q}=O(n)$ for some $q>3$.  These estimates are established in~\cite{KM16} for $\gamma\in[0,\frac{2}{11})$.  An improvement in~\cite{KKMprep} covers
$\gamma\in[0,\frac14)$ and this is known to be sharp~\cite{MN08,M09b}.
Hence the parameter regime $\gamma\in[\frac14,\frac12)$ is beyond the H\"older rough path theory.  In contrast, the $p$-variation rough path theory described here requires the moment estimates only for some $q>1$ and~\cite{KKMprep} applies for all $\gamma\in[0,\frac12)$.  Hence we are able to prove homogenization theorems in the full range
$\gamma\in[0,\frac12)$.

\bigskip

The remainder of this article is organized as follows. In Section~\ref{sec:ER}, we discuss the WIP and chaotic dynamics, and several situations of homogenization where rough path theory is not required. In Section~\ref{sec:RPs}, we introduce the parts of rough path theory required in the Brownian motion setting of this paper. This is applied to fast-slow systems in Section~\ref{sec:appl}. In Section~\ref{sec:exts}, we mention extensions and related work.

\bigskip

{\bf Acknowledgements:} I.C. is funded by a Junior Research Fellowship of St John's College, Oxford.
P.K.F. acknowledges partial support from the ERC, CoG-683164, the Einstein Foundation Berlin, and DFG research unit FOR2402.
A.K. and I.M. acknowledge partial support from the European Advanced Grant StochExtHomog (ERC AdG 320977).
H.Z. is supported by the Chinese National Postdoctoral Program for Innovative Talents No: BX20180075.
H.Z. thanks the Institute für Mathematik, TU Berlin, for its hospitality.

\section{Emergence of randomness in deterministic dynamical systems}
\label{sec:ER}

In this section, we review a simplified situation where the ordinary weak invariance principle (see below) suffices, and rough path theory is not required.

\subsection{The weak invariance principle} \label{sec:WIP}

Consider a family of stochastic processes indexed by $\eps \in (0,1)$, say $W_\eps = W_\eps (t,\omega)$ with values in $\R^m$. We are interested in convergence of the respective laws.
In the case of continuous sample paths (including smooth or piecewise linear) we say that the {\it weak invariance principle} (WIP) holds if
\begin{equation*}  \label{equ:WIPinC}
          W_{\eps}  \to_{w} W  \text{ in $C([0,1],\R^m) $ as $\eps \to 0$}\;,
\end{equation*}
where $W$ is an $m$-dimension Brownian motion with covariance matrix $\Sigma$; in the case of c{\`a}dl{\`a}g sample paths (including piecewise constant) we mean
\begin{equation*}  \label{equ:WIPinD}
          W_{\eps}  \to_{w} W  \text{ in $D([0,1],\R^m) $ as $\eps \to 0$}\;,
\end{equation*}
where $C$ resp.\ $D$ denotes the space of continuous resp.\ c{\`a}dl{\`a}g paths, equipped with the uniform topology.\footnote{Since our limit processes here - a Brownian motion - is continuous,
 there is no need to work with the Skorokhod topology on $D$.} For notational simplicity only, assume ($W_\eps)$ are defined on a common probability space $(\Omega,F,\lambda)$; we then write $W_{\eps}  \to_\lambda W$ to indicate convergence in law, i.e. $\E_\lambda [ f  (W_\eps ) ]\to E_\lambda [ f (W) ]$ for all bounded continuous functionals.

In many cases, one has convergence of second moments. This allows to compute the covariance of the limiting Brownian motion,
\begin{equation} \label{Equ:C2M}
  \Sigma = \E ( W (1) \otimes W (1) )
  = \lim_{\eps \to 0} \E_\lambda (W_{\eps} (1) \otimes W_\eps(1))
  \;.
\end{equation}

The WIP is also known as the {\it functional central limit theorem}, with the CLT for finite-dimensional distributions as a trivial consequence. Conversely, the CLT for f.d.d.\ together with tightness gives the WIP.

{\it Donsker's invariance principle}~\cite{Donsker51} is the prototype of a WIP: consider a centered $m$-dimensional random walk $Z_n := \xi_1 + \cdots + \xi_n$, with $\R^m$-valued IID increments of zero mean and finite covariance $\Sigma$. Extend to either a continuous piecewise linear process or c{\`a}dl{\`a}g piecewise constant process $(Z_t: t \ge 0)$. Then the WIP holds for the rescaled random walk
$$
Z^\eps_t := \eps Z_{t / \eps^2} \ ,
$$
and the limiting Brownian motion has covariance $\Sigma$. This result has an important generalization to a {\it (functional) martingale CLT}: using similar notation, assume $(Z_n)$ is a zero mean $L^2$-martingale with stationary and ergodic increments $(\xi_i)$. Then, with the identical rescaling, the WIP holds true, with convergence of second moments~\cite{Brown71} (or e.g.~\cite[Thm.~18.3]{Billingsley}).

Another interesting example is given by {\it physical Brownian motion} with positive mass $\eps^2 >0$ and friction matrix $M$, where the trajectory is given by
$$
           X^\eps_t := \eps \int_0^{t / \eps^{2}} Y_s \, ds
$$
and $Y$ follows an $m$-dimensional OU process, $dY = - MY dt + dB, \ Y_0 = y_0$. Here, $M$ is an $m \times m$-matrix whose spectrum has positive real part, and $B$ is an $m$-dimensional standard Brownian motion. One checks without difficulties~\cite{PavliotisStuart08, FGL15} that a WIP holds, even in the sense of weak convergence in the H\"older space $C^\alpha ([0,1],\R^m)$ with any $\alpha < 1/2$. The covariance matrix of the limiting Brownian is given by $\Sigma = M^{-1} (M^{-1})^T$, as can be seen from the Newton dynamics $\eps^2 \ddot X^\eps = - M \dot X^\eps + \dot B$ with white noise $\dot B$.

\bigskip

\noindent
Finally, {\it sufficiently chaotic deterministic dynamical systems} are a rich source of WIPs.
To fix ideas, consider a compact Riemannian manifold $M$ with a Lipschitz vector field $g$ and corresponding flow $g_t$, for which there is an ergodic, invariant Borel probability measure $\mu$ on $M$. We regard $(g_t)$ as an $M$-valued stochastic process, given by $ g_t( y_0 )$ with initial condition $y_0$ distributed according to $\lambda$, another Borel probability measure on $M$. (It is possible but not necessary to have $\lambda = \mu$.) Consider further a suitable observable $v: M\to\R^m$ with $\E_\mu v = 0$.
A family of $C^1$-processes $(\Wve)_{\eps >0}$, with values in $\R^m$, is then given by
\begin{align*}
 \Wve(t)=\eps\int_0^{t\eps^{-2}}v \circ g_{s}\,ds \ .
\end{align*}
As will be reviewed in Section~\ref{sec:chaos} below, also in a discrete time setting, in many situations a WIP holds. That is,
$$
          \Wve  \to_\lambda W     \text{ in $C([0,1],\R^m) $ as $\eps \to 0$}\ .
$$
Typically one also has convergence of second moments, so that $W$ is a Brownian motion with covariance $\Sigma$ given by \eqref{Equ:C2M}.
Under (somewhat restrictive) assumptions on the decay of correlations, this can be simplified
to a Green-Kubo type formula
$$
    \Sigma = \int_0^\infty  \E_\mu \{ v \otimes ( v \circ g_s ) + (v \circ g_s ) \otimes v \} ds \ .
$$

\subsection{First applications to fast-slow systems}
\label{sec:RPappl}

In the setting of deterministic, sufficiently chaotic dynamical systems discussed in the previous paragraph, Melbourne--Stuart~\cite{MS11} consider the fast-slow system posed on $\R^d\times M$ (with $m=d$),
\begin{align*}
\dot x_\varepsilon & = a(x_\varepsilon,y_\varepsilon) + \varepsilon^{-1}v(y_\varepsilon) \;, \qquad
\dot y_\varepsilon  = \varepsilon^{-2}g(y_\varepsilon) \;,
\end{align*}
with deterministic initial data $x_\varepsilon (0) = x_0$ and $y_\varepsilon (0)$ sampled randomly with probability $\lambda$.
We wish to study the limiting dynamics of the slow variable $x_\eps$.
Assuming for simplicity $a(x,y)=a(x)$, the basic observation is to rewrite
$$
     \dot x_\varepsilon = a(x_\varepsilon) + \dot \Wve \ .
$$
We see that the noise $\Wve$ enters the equation in an additive fashion and one checks without difficulty that the ``It\^o-map'' $\Wve \mapsto x_\eps$ extends continuously (w.r.t.\ uniform convergence) to any continuous noise path. Now assume validity of a WIP, i.e. $ \Wve \to_\lambda W$. Then, together with continuity of the It{\^o}-map, one obtains the desired limiting SDE dynamics of the slow variable as
$$
    dX = dW + a(X)dt\;.
$$
In the general case when $a$ depends on $x_\eps$ and $y_\eps$, the drift term is given by $\bar a (x) = \int_M a (x,y) \, d\mu (y)$

In subsequent work, Gottwald--Melbourne~\cite{GM13} consider the one-dimensional case $d=m=1$ with
\begin{align*}
\dot x_\varepsilon &= a(x_\varepsilon, y_\varepsilon) + \varepsilon^{-1}b(x_\varepsilon) v(y_\varepsilon) \;,
\qquad
\dot y_\varepsilon = \varepsilon^{-2}g(y_\varepsilon)\;.
\end{align*}
Again, taking $a(x,y)=a(x)$ for simplicity, the limiting SDE turns out to be of Stratonovich form
$$
       dX = a(X)dt + b(X) \circ dW \; .
$$
The essence of the proof is a robust representation of such SDEs.
Indeed, taking $a \equiv 0$ for notational simplicity, an application of the (first order) Stratonovich chain rule exhibits the explicit solution as $X_t = e^{W_t b}(X_0)$, where $e^{W_tb}$ denotes the flow at ``time'' $W_t \in \R$ along the vector field $b$; this clearly depends continuously on $X_0$ and $W$ w.r.t.\ uniform convergence.
Hence, as in the additive case, the problem is reduced to having a WIP. This line of reasoning
can be pushed a little further, namely to the case
$\dot x_\varepsilon = a(x_\varepsilon, y_\varepsilon) + \varepsilon^{-1} V(x_\varepsilon) v(y_\varepsilon)$
with commuting vector fields $V=(V_1, ..., V_m)$, a.k.a.\ the Doss--Sussmann method, but
fails for general vector fields, not to mention the non-product case when $V(x)v(y)$ is replaced by $b(x,y)$. This is a fundamental problem which is addressed by Lyons' theory of rough paths.

Gottwald-Melbourne~\cite{GM13} consider also discrete time fast-slow systems
posed on $\R^d\times M$,
\[
X^{(n)}_{j+1} = X^{(n)}_j + n^{-1}a(X_j^{(n)},Y_j) + n^{-1/2}b(X_j^{(n)})v(Y_j)\;,
\quad
Y_{j+1} = TY_j\;,
\]
Again we suppose for notational simplicity that $a(x,y)=a(x)$.
We continue to suppose that $\mu$ is an ergodic $T$-invariant probability measure on $M$ and that $\E_\mu v=0$.  Also, $\lambda$ is another probability measure on $M$.
Recall that $x_{\eps}(t)= X^{\floor{1/\eps^2}}_{\floor{t/\eps^2} }$
and assume validity of a WIP, i.e.\ $W_\eps \to_\lambda W$.
When $b\equiv1$, it is shown in~\cite{GM13} that
$x_\eps\to_\lambda X$ where $dX=a(X)\,dt+dW$.
For $d=m=1$, under additional mixing assumptions it is shown that
$x_\eps\to_\lambda X$ where $dX=\tilde a(X)\,dt+b(X)\,dW$
with
\[
\tilde a(x)=a(x)+b(x)b'(x)\sum_{n=1}^\infty \E_\mu ( v\;v\circ T^n).
\]

\subsection{Chaotic dynamics: CLT and the WIP}
\label{sec:chaos}

In this subsection, we describe various classes of dynamical systems with good statistical properties, focusing attention on the CLT and WIP.

Somewhat in contrast to rough path theory,
the ergodic theory of smooth dynamical systems is much simpler for discrete time than for continuous time -- indeed the continuous time theory proceeds by reducing to the discrete time case.
Also, the simplest examples are noninvertible.
The reasons behind this are roughly as follows.
Since the papers of Anosov~\cite{Anosov67} and Smale~\cite{Smale67}, it is has been understood that the way to study dynamical systems is to exploit expansion and contraction properties.
The simplest systems are uniformly expanding; these are necessarily discrete time and noninvertible.
Anosov and Axiom~A (uniformly hyperbolic) diffeomorphisms have uniformly contracting and expanding directions.
Anosov and Axiom~A (uniformly hyperbolic) flows have a neutral time direction and are uniformly contracting and expanding in the remaining directions.
The neutral direction makes flows much harder to study.
The mixing properties of uniformly hyperbolic flows are still poorly understood (see for example the review in~\cite{rapid}); fortunately the CLT and WIP do not rely on mixing.

Accordingly, we consider in turn expanding maps, hyperbolic diffeomorphisms, and hyperbolic flows, in Subsections~\ref{sec:NUE},~\ref{sec:NUH} and~\ref{sec:flow}
respectively.
This includes the uniform cases mentioned in the previous paragraph, but also dynamical systems that are nonuniformly expanding/hyperbolic, which is crucial for incorporating large classes of examples.

\subsubsection{Uniformly and nonuniformly expanding maps} \label{sec:NUE}

%\begin{defi}
%We say that a stochastic process $X$ satisfies the \emph{almost sure invariance principle} (ASIP)
%if for some $\beta \in (0,1/2)$, enlarging if necessary the probability space,
%there exists a Brownian motion $W$ such
%that
%\[
%  X_t = W_t + O(t^\beta)
%  \quad \text{almost surely}.
%\]
%\end{defi}

The CLT and WIP are proved in~\cite{HofbauerKeller82,Keller85} for large classes of dynamical systems
(in fact, they prove a stronger statistical property, known as the almost sure invariance principle).  For recent developments in this direction, see~\cite{CunyMerlevede15,KKM16} and references therein.
In particular, the CLT and WIP hold for smooth uniformly expanding maps and for systems modelled by Young towers with summable decay of correlations~\cite{Young99}, which provide a rich source of examples including the intermittent maps~\eqref{eq:LSV}.

Here we review the results in various situations, focusing on various issues that are of importance for fast-slow systems: CLT, WIP, covariance matrices, nondegeneracy.
Also, we mention the notions of spectral decomposition, mixing up to a finite cycle, basins of attraction, and strong distributional convergence, which are necessary for understanding how the theory is applied.
\smallskip

{\it Smooth uniformly expanding maps.}
The simplest chaotic dynamical system is the doubling map
$T:M\to M$, $M=[0,1]$, given by $Ty=2y\bmod 1$.
More generally, let $T:M\to M$ be a $C^2$ map on a
compact Riemannian manifold $M$ and let $\cB$ be the $\sigma$-algebra of Borel sets.
The map is {\em uniformly expanding} if
there are constants $C>0$, $L>1$ such that
$\|DT^n|_y z\|\ge CL^n\|z\|$ for all $y\in M$, $z\in T_yM$.
By~\cite{KrzSzlenk69},
there is a unique ergodic $T$-invariant Borel probability measure~$\mu$ on $M$ equivalent to the volume measure.
(Recall that $\mu$ is $T$-invariant
if $\mu(T^{-1}B)=\mu(B)$ for all $B\in \cB$, and is ergodic if $\mu(B)=0$ or $\mu(B)=1$ for all $B\in \cB$ with $TB\subset B$.)
By Birkhoff's ergodic theorem~\cite{Birkhoff31}
(an extension of the strong law of large numbers)
the sum
\(
v_n=\sum_{0 \leq j < n} v\circ T^j
\)
satisfies $n^{-1}v_n\to \int_M v\,d\mu$ a.e.\ for all $v\in L^1$.

To make further progress it is necessary to impose some regularity on the observable~$v$; the CLT fails in general for continuous observables.
Hence, we suppose that $v$ is H\"older.
Specifically, fix $\kappa\in(0,1)$ and
let $C^\kappa_0(M)$ be the space of $C^\kappa$ observables
$v:M\to\R$ with $\int_M v\,d\mu=0$.
It is well known~\cite{Bowen75,Ruelle78,Sinai72} that
there are constants $\gamma\in(0,1)$ and $C>0$ depending only on $T$ and $\kappa$ such that
\begin{align} \label{eq:decay}
\Big|\int_M v\,w\circ T^n\,d\mu\Big|\le C\gamma^n\|v\|_{C^\kappa} \|w\|_1
\quad\text{for all $v\in C^\kappa_0(M)$, $w\in L^1$, $n\ge1$}.
\end{align}
An immediate consequence is that the limit
$$
\sigma^2 := \lim_{n\to\infty}n^{-1}\int_M v_n^2\,d\mu\;,
$$
exists
and that $$\sigma^2=\int_M v^2\,d\mu+2\sum_{n=1}^\infty \int_M v\,v\circ T^n\,d\mu \ . $$

For $1\le p\le\infty$, we recall that the Koopman operator $U:L^p\to L^p$ is given by $Uv=v\circ T$ and that the transfer operator $P:L^q\to L^q$ is given by
$\int_M Pv\;w\,d\mu=\int_M v\;w\circ T\,d\mu$ for $v\in L^q$, $w\in L^p$, where $p^{-1}+q^{-1}=1$.
These operators satisfy $\|U\|_p=1$ and
$\|P\|_p\le 1$ for $1\le p\le\infty$.  In addition, $PU=I$ and
$UP=\E(\,\cdot \,| T^{-1}\cB)$.
Note that property~\eqref{eq:decay} is equivalent to
\begin{align} \label{eq:P}
\|P^nv\|_\infty \le C\gamma^n\|v\|_{C^\kappa}
\quad\text{for all $v\in C^\kappa_0(M)$, $n\ge1$}.
\end{align}

Following the classical approach of Gordin~\cite{Gordin69}, we define $\chi=\sum_{j=1}^\infty P^jv$ and
$m=v-\chi\circ T+\chi$.  By~\eqref{eq:P}, $m,\,\chi\in L^\infty$.
It follows from the definitions that $m\in\ker P$ and hence
that $n^{-1}\int_M m_n^2\,d\mu=\int_M m^2\,d\mu$ for all $n$.
Since
$$v_n-m_n=\chi\circ T^n-\chi \in L^\infty \ , $$
it follows that $\int_M m^2\,d\mu=\sigma^2$.

Moreover,
$\E(m|T^{-1}\cB)=UPm=0$, so $\{m\circ T^n,\,n\ge0\}$ is an $L^\infty$ stationary ergodic reverse martingale difference sequence.  Hence, standard martingale limit theorems apply.  In particular, by~\cite{Billingsley61,Mcleish74} we obtain the CLT for $m$ and thereby~$v$:
\[
n^{-1/2}v_n\to_\mu N(0,\sigma^2) \quad \text{as $n\to\infty$}.
\]
We refer to the decomposition $v=m+\chi\circ T-\chi$ as an $L^\infty$ {\em martingale-coboundary decomposition}, since $m$ is a reverse martingale increment.  The coboundary term $\chi\circ T-\chi \in L^\infty$ telescopes under iteration and therefore is often negligible.
Next, define the process $W_n\in C[0,1]$ by setting
$V_n(t)=n^{-1/2}v_{nt}$ for $t=0,1/n,2/n\dots$ and linearly interpolating.
By~\cite{Brown71,Mcleish74} we obtain the WIP:
\[
W_n\to_\mu W \quad\text{in $C[0,1]$ as $n\to\infty$,}
\]
where $W$ is a Brownian motion with variance $\sigma^2$.

The CLT and WIP are said to be {\em degenerate} if $\sigma^2=0$.
We now show that this is extremely rare.
Since $v=m+\chi\circ T-\chi$ and $\sigma^2=\int_M m^2\,d\mu$, we obtain that $\sigma^2=0$ if and only if $v=\chi\circ T-\chi$ where $\chi\in L^\infty$.
Moreover, the series $\chi=\sum_{n=1}^\infty P^nv$ converges in $C^\kappa$ (see for example~\cite{Ruelle89}), so in particular $\chi$ is continuous.
Let $C^\kappa_{\rm deg}$ consist of observables $v\in C^\kappa_0(M)$ with $\sigma^2=0$.

\begin{prop} \label{prop:deg}
$C^\kappa_{\rm deg}$ is a closed, linear subspace of infinite codimension in $C^\kappa_0(M)$.
\end{prop}

\begin{proof}
Suppose that $v\in C^\kappa_{\rm deg}$ so $v=\chi\circ T-\chi$ where $\chi$ is continuous.  Iterating, we obtain $v_k=\chi\circ T^k-\chi$.
Hence if $y\in M$ is a period $k$ point, i.e.\ $T^ky=y$, then
$v_k(y)=0$.  Since periodic points are dense in $M$~\cite{Smale67} we obtain infinitely many linear constraints on~$v$.
\end{proof}

Analogous results hold for vector-valued observables $v:M\to\R^m$.
Let $v\in C^\kappa_0(M,\R^m)$.
Then
\begin{align} \label{eq:Sigma1}
\lim_{n\to\infty}n^{-1}\int_M v_n\otimes v_n\,d\mu=\Sigma,
\end{align}
where
$\Sigma\in\R^m\otimes \R^m$ is symmetric and positive semidefinite,
and
\begin{align} \label{eq:Sigma2}
\Sigma=\int_M v\otimes v\,d\mu+\sum_{n=1}^\infty \int_M
\{v\otimes(v\circ T^n)+(v\circ T^n)\otimes v\}\, d\mu.
\end{align}
Define $v_n\in\R^m$ and $W_n\in C([0,1],\R^m)$ as before.
By the above results, $n^{-1/2}c^Tv_n$ converges in distribution to a normal distribution with variance $c^T\Sigma c$ for each $c\in\R^m$, and hence by
Cramer-Wold we obtain the multi-dimensional CLT
$n^{-1/2}v_n\to_d N(0,\Sigma)$.
Similarly, $W_n\to_w W$ in $C([0,1],\R^m)$ where $W$ is $m$-dimensional Brownian motion with covariance $\Sigma$.
Finally $c^T\Sigma c=0$ for $c\in\R^m$ if and only if $c^Tv\in C^\kappa_{\rm deg}$.
Hence the degenerate case $\det\Sigma=0$ occurs only on a closed subspace of infinite codimension.

Since the CLT is a consequence of the WIP, generally we only mention the WIP in the remainder of this subsection.

Returning to the ergodic theorem, if $v\in L^1$, then $n^{-1}v_n(y_0)\to \int_M v\,d\mu$ for $\mu$ almost every initial condition $y_0\in M$.
Since $\mu$ is equivalent to volume, we could equally choose the initial condition $y_0$ randomly with respect to volume, which is perhaps more natural since volume is the intrinsic measure on $M$.
Similar considerations apply to the WIP.  Based on ideas of~\cite{Eagleson76},
it follows from~\cite[Cor.~2]{Zweimuller07} that
if $v\in C^\kappa_0(M,\R^m)$ then
$W_n\to_\lambda W$ in $C([0,1],\R^m)$ for every absolutely continuous Borel probability measure $\lambda$ (including $\mu$ and volume as special cases).
This property is often called {\em strong distributional convergence}~\cite{Zweimuller07}.
Of course $C^\kappa_0(M,\R^m)$ is defined using $\mu$ regardless of the choice of $\lambda$.
\smallskip

{\it Piecewise expanding maps.}
There are numerous extensions of the above arguments in various directions.
For example,
Keller~\cite[Thm.~3.5]{Keller85} considers piecewise $C^{1+\eps}$ transformations $T:M\to M$, $M=[0,1]$,  with finitely many monotone branches and $|T'|\ge L$ for some $L>1$.  There exists an ergodic $T$-invariant absolutely continuous probability measure (acip) $\mu$.  Let $\Lambda=\supp\mu$.
Recall that $\Lambda$ is {\em mixing} if
$\lim_{n\to\infty}\mu(T^{-n}A\cap B)=\mu(A)\mu(B)$ for all
measurable sets $A,B\subset \Lambda$.
In this case, by~\cite[Thm.~3.3]{Keller85},
condition~\eqref{eq:decay} holds (with $M$ replaced by $\Lambda$).
Hence we obtain the WIP for all $v\in C^\kappa_0(\Lambda,\R^m)$
with $\Sigma$ given as in~\eqref{eq:Sigma1} and~\eqref{eq:Sigma2}.
Also $\det\Sigma=0$ if and only if there exists $c\in\R^m$ such that
$c^Tv=\chi\circ T-\chi$ for some $\chi:\Lambda\to\R$ in $L^\infty$.

If $\Lambda$ is not mixing, then condition~\eqref{eq:decay} fails.
Nevertheless, by~\cite[Thm.~3.3]{Keller85} $\Lambda$ is {\em mixing up to a finite cycle}: we can write $\Lambda$ as a disjoint union $\Lambda=A_1\cup\cdots\cup A_k$ for some $k\ge2$ such that $T$ permutes the $A_j$ cyclically and $T^k:A_j\to A_j$ is mixing with respect to $\mu|A_j$
for each $j$.  Moreover, condition~\eqref{eq:decay} holds for the map
$T^k:A_j\to A_j$.  It is easily verified that the WIP goes through for $T:\Lambda\to\Lambda$ and that the limit formula~\eqref{eq:Sigma1} for $\Sigma$ remains valid.  (Of course in the nonmixing case,~\eqref{eq:Sigma2} no longer makes sense.)

The {\em basin of attraction} of the ergodic probability measure $\mu$ is defined as
\[
B_\mu=\{y\in M:
\lim_{n\to\infty} n^{-1}v_n(y)={\textstyle\int}_M v\,d\mu
\quad\text{for all $v:M\to\R$ continuous}\}.
\]
(The ergodic theorem guarantees that, modulo a zero measure set, $\supp\mu\subset B_\mu$, but in general $B_\mu$ can be much larger.)
The acip $\mu$ need not be unique but by~\cite[Thm.~3.3]{Keller85} there is {\em a spectral decomposition}: there exist finitely many absolutely continuous ergodic invariant probability measures $\mu_1,\dots,\mu_k$ such that
$\Leb(B_{\mu_1}\cup\cdots\cup B_{\mu_k})=1$
and the results described above for $\mu$ hold separately for each of $\mu_1,\dots,\mu_k$.

For related results on $C^2$ one-dimensional maps with infinitely many branches, we refer to~\cite{Rychlik83}.  For higher-dimensional piecewise smooth maps, see for example~\cite{BuzziMaume02,Saussol00}.
Again there is a spectral decomposition into finitely many attractors which are mixing up to a finite cycle.  After restricting to an appropriate subset and considering a suitable iterate of $T$, condition~\eqref{eq:decay} holds and we obtain the WIP etc as described above.

In general, extra work is required to deduce that degeneracy is infinite codimension as in Proposition~\ref{prop:deg}.  We note that the approach in~\cite{BuzziMaume02} fits within the Young tower approach of~\cite{Young98,Young99} where it is possible to recover Proposition~\ref{prop:deg} as described below.
\smallskip

{\it Nonuniformly expanding maps.}
An important method for studying nonuniformly expanding maps $T:M\to M$ is to construct a {\em Young tower} as in~\cite{Young99}.  This incorporates the maps~\eqref{eq:LSV} discussed in the introduction.

Let $M$ be a bounded metric space with
finite Borel measure $\rho$ and let $T:M\to M$ be a nonsingular
transformation ($\rho(T^{-1}B)=0$ if and only if $\rho(B)=0$ for $B\in\cB$).
Let $Y\subset M$ be a subset of positive measure, and
let~$\alpha$ be an at most countable measurable partition
of $Y$ with $\rho(a)>0$ for all $a\in\alpha$.    We suppose that there is an
integrable
{\em return time} function $\tau:Y\to\Z^+$, constant on each $a$ with
value $\tau(a)\ge1$, and constants $L>1$, $\kappa\in(0,1)$, $C_0>0$,
such that for each $a\in\alpha$,
\begin{itemize}
\item[(1)] $F=T^\tau$ restricts to a (measure-theoretic) bijection from $a$  onto $Y$.
\item[(2)] $d(Fx,Fy)\ge L d(x,y)$ for all $x,y\in a$.
\item[(3)] $d(T^\ell x,T^\ell y)\le C_0d(Fx,Fy)$ for all $x,y\in a$,
$0\le \ell <\tau(a)$.
\item[(4)] $\zeta_0=\frac{d\rho|_Y}{d\rho|_Y\circ F}$
satisfies $|\log \zeta_0(x)-\log \zeta_0(y)|\le C_0d(Fx,Fy)^\kappa$ for all
\mbox{$x,y\in a$}.
\end{itemize}
The {\em induced map} $F=T^\tau:Y\to Y$ has
a unique acip $\mu_Y$.

\begin{rem}For the intermittent maps~\eqref{eq:LSV}, we can take $Y=[\frac12,1]$ and we can choose $\tau$ to be the first return to $Y$.
In general, it is not required that $\tau$ is the first return time to $Y$.
\end{rem}

Define the Young tower~\cite{Young99},
$\Delta=\{(y,\ell)\in Y\times\Z:0\le\ell\le \tau(y)-1\}$,
and the tower map
\begin{align} \label{eq:Delta}
f:\Delta\to\Delta, \qquad
f(y,\ell)=\begin{cases} (y,\ell+1), & \ell\le \tau(y)-2
\\ (Fy,0), & \ell=\tau(y)-1 \end{cases}.
\end{align}
The projection
$\pi_\Delta:\Delta\to \Lambda$, $\pi_\Delta(y,\ell)=T^\ell y$, defines a semiconjugacy
from $f$ to $T$.  Define the ergodic acip
$\mu_\Delta=\mu_Y\times\{{\rm counting}\}/\int_Y \tau\,d\mu_Y$ for $f:\Delta\to\Delta$.
Then
$\mu=(\pi_\Delta)_*\mu_\Delta$ is an ergodic acip for $T:M\to M$
and $\mu$ is mixing up to a finite cycle.

Young~\cite{Young99} proved that if $\mu$ is mixing and
$\mu_Y(y\in Y:\tau(y)>n)=O(n^{-(\beta+1)})$ for some $\beta>0$,
then
\begin{align} \label{eq:poly}
\Big|\int_M v\,w\circ T^n\,d\mu\Big|\le Cn^{-\beta}\|v\|_{C^\kappa} \|w\|_\infty
\quad\text{for all $v\in C^\kappa_0(M)$, $w\in L^\infty$, $n\ge1$},
\end{align}
In particular, $\beta>1$ corresponds to {\em summable decay of correlations.}
(For the maps~\eqref{eq:LSV}, $\beta=\gamma^{-1}-1$, so $\beta > 1$ corresponds to $\gamma < \frac12$.)
Equivalently,
$\|P^nv\|_1\le Cn^{-\beta}\|v\|_{C^\kappa}$ and by interpolation
$\|P^nv\|_p\le C^{1/p}n^{-\beta/p}\|v\|_{C^\kappa}$ for all $p\ge1$.

For $\beta>1$, we have that $\|P^nv\|_1$ is summable for $v\in C^\kappa_0(M,\R^m)$,
and a standard calculation shows that
formulas~\eqref{eq:Sigma1} and~\eqref{eq:Sigma2} for $\Sigma$ hold.
Also, the series $\chi=\sum_{n=1}^\infty P^nv$ converges in $L^p$ for all $p<\beta$ and we obtain an
$L^p$ martingale-coboundary decomposition $v=m+\chi\circ T-\chi$.
For $\beta>2$, we have $m,\,\chi\in L^2$ and the WIP follows.
With extra work it can be shown that the WIP holds for all $\beta>1$.  We refer to~\cite{KipnisVaradhan86,Liverani96,MaxwellWoodroofe00,TyranKaminska05} for further details.  See also~\cite{KKM16,MN05}.  By~\cite[Rem.~2.11]{MN05}, the degenerate case $\det\Sigma=0$ is infinite codimension in the sense of Proposition~\ref{prop:deg}.

In the case where $\mu$ is mixing only up to a finite cycle, the WIP etc go through unchanged, except that formula~\eqref{eq:Sigma2} does not make sense.

\subsubsection{Hyperbolic diffeomorphisms} \label{sec:NUH}

A WIP for Axiom~A diffeomorphisms can be found in~\cite{DenkerPhilipp84}.
The WIP is also well-known  to hold for
systems modelled by Young towers with exponential tails~\cite{Young98} as well as those with summable decay of correlations (for an explicit and completely general argument, see~\cite{MV16}).  This is a very flexible setting that covers large classes of nonuniformly hyperbolic diffeomorphisms (with singularities).

The results for hyperbolic diffeomorphisms $T:M\to M$ are similar to those in
Subsection~\ref{sec:NUE}, subject to
two complications.
The first complication affects the proofs.
Since $T$ is invertible,
the transfer operator~$P$ is an isometry on $L^q$ for all $q$.
In particular, $\ker P=\{0\}$.  Hence the approach in Subsection~\ref{sec:NUE} cannot be applied directly.
The method for getting around this is rather convoluted and is described at the end of this subsection.

The second complication affects the statement of the results.
Typically, the invariant measures of interest are supported on zero volume sets and hence there are no acips.
We say that $\mu$ is a {\em physical measure} if
the basin of attraction $B_\mu$ has positive volume.
(This is automatic for acips but is an extra assumption now.)  Let $\Vol$ denote the normalized volume on $B_\mu$.

There is an important class of physical measures $\mu$, known as Sinai-Ruelle-Bowen (SRB) measures~\cite{Young02}, for which the WIP with respect to $\Vol$ (and hence, by strong distributional convergence, every absolutely continuous probability measure $\lambda$ on $B_\mu$) follows from the WIP with respect to~$\mu$.
Hence it is natural to consider observables $v$ with $\int_\Lambda v\,d\mu=0$ and to ask that $W_n\to_\lambda W$ for absolutely continuous probability measures $\lambda$ on $B_\mu$.
\smallskip

{\it Axiom~A diffeomorphisms.}
Let $M$ be a compact Riemannian manifold.
A $C^2$ diffeomorphism $T:M\to M$ is said to be {\em Anosov}~\cite{Anosov67}
if there is a continuous $DT$-invariant splitting
$TM=E^s\oplus E^u$ (into stable and unstable directions) where $\|DT^n|E^s\|\le Ca^n$
and $\|DT^{-n}|E^u\|\le Ca^n$ for $n\ge1$.
Here $C>0$ and $a\in(0,1)$ are constants.

Smale~\cite{Smale67} introduced the notion of {\em Axiom~A} diffeomorphism extending the definition in~\cite{Anosov67}.
Since we are interested in SRB measures, we restrict attention to attracting sets, bypassing the full definitions in~\cite{Smale67}.
Recall that a closed $T$-invariant set $\Lambda\subset M$
is {\em attracting} if there is a neighbourhood $U$ of $\Lambda$ such that $\lim_{n\to\infty}\dist(T^ny,\Lambda)=0$ for all $y\in U$.
An attracting set is called {\em Axiom~A} if
there is a continuous $DT$-invariant splitting
$T_\Lambda M=E^s\oplus E^u$ over $\Lambda$, again with the properties
$\|DT^n|E^s\|\le Ca^n$
and $\|DT^{-n}|E^u\|\le Ca^n$ for $n\ge1$.
To avoid trivialities, we suppose that $\dim E^u_y\ge1$ for all $y\in\Lambda$.
(We allow $\dim E^s_y=0$ though this is just the uniformly expanding case.)

By~\cite{Smale67}, there is a spectral decomposition of $\Lambda$ into finitely many attracting sets, called {\em Axiom~A attractors} with the property that none of them can be decomposed further.
Moreover, periodic points are dense in $\Lambda$.

If $\Lambda$ is an Axiom~A attractor, then by~\cite{Bowen75,Ruelle78,Sinai72} there is a unique ergodic invariant probability measure $\mu$ on $\Lambda$ such that $\Leb(B_\mu) > 0$.
Moreover, $\mu$ is mixing up to a finite cycle.

All the results described in Subsection~\ref{sec:NUE} for uniformly expanding maps hold for Axiom~A attractors.
Specifically, let $C^\kappa_0(\Lambda,\R^m)$ denote the space of
$C^\kappa$ observables $v:\Lambda\to\R^m$ with $\int_\Lambda v\,d\mu=0$.
Then the WIP holds on $(\Lambda,\mu)$ and $(B_\mu,\lambda)$ with $\Sigma$ satisfying formula~\eqref{eq:Sigma1}.  Moreover $C^\kappa_{\rm deg}=\{v\in C^\kappa_0(\Lambda, \R^m):\det\Sigma=0\}$ is a closed subspace of infinite codimension in $C^\kappa_0(\Lambda, \R^m)$.
If in addition $\mu$ is mixing, then formula~\eqref{eq:Sigma2} holds.
\smallskip

{\it Nonuniformly hyperbolic diffeomorphisms and Young towers.}
A large class of attractors $\Lambda$ for nonuniformly hyperbolic diffeomorphisms (with singularities) $T:M\to M$ can be modelled by two-sided Young towers with exponential tails~\cite{Young98} and subexponential tails~\cite{Young99}.
The Young tower set up covers numerous classes of examples as surveyed
in~\cite{ChernovYoung00,Young98,Young02,Young17} including Axiom~A attractors, Lorentz gases,
H\'enon-like attractors~\cite{BenedicksYoung00},
and intermittent solenoids~\cite{MV16}.  See also~\cite{AFLV11,AlvesLuzzattoPinheiro05,AlvesPinheiro10}.

We end this subsection with a very rough sketch of the method of proof of the WIP for Young towers.   This includes the Axiom~A attractors as a special case for which standard references are~\cite{Bowen75,ParryPollicott90}.
The idea is again to induce to a map $F=T^\tau:Y\to Y$ that is a uniformly hyperbolic transformation with countable partition and full branches, as described in Young~\cite{Young98}, with an integrable inducing time $\tau:Y\to\Z^+$ that is constant on partition elements.  (Again $\tau$ is not necessarily the first return time.)  The construction in~\cite{Young98} ensures that there exists an SRB measure $\mu_Y$ for $F$.
Starting from $F$ and $\tau$, we construct a ``two-sided'' Young tower
$f:\Delta\to\Delta$ as in~\eqref{eq:Delta}
with ergodic invariant probability
$\mu_\Delta=\mu_Y\times\{{\rm counting}\}/\int_Y \tau\,d\mu_Y$.
The projection
$\pi_\Delta:\Delta\to \Lambda$, $\pi_\Delta(y,\ell)=T^\ell y$, defines a semiconjugacy
from $f$ to $T$, and
$\mu=(\pi_\Delta)_*\mu_\Delta$ is the desired SRB measure for $T:M\to M$.
Moreover, $\mu$ is mixing up to a finite cycle.

Given $v\in C^\kappa_0(M,\R^m)$, we define the lifted observable $\hat v=v\circ\pi_\Delta:\Delta\to\R^m$.  It suffices to work from now on with $\hat v$.

Next, there is a quotienting procedure which projects out the stable directions reducing to an expanding map.  Formally, this consists of a ``uniformly expanding'' map
$\bar F:\bar Y\to\bar Y$ and a projection $\pi:Y\to\bar Y$ such that
$\bar F\circ\pi=\pi\circ F$ and such that
$\tau(y)=\tau(y')$ whenever $\pi y=\pi y'$.   In particular, $\tau$ projects to a well-defined return time $\tau:\bar Y\to\Z^+$.
Using $\bar F$ and $\tau$ we construct
a ``one-sided'' Young tower $\bar f:\bar\Delta\to\bar\Delta$.
The projection $\pi$ extends to $\pi:\Delta\to\bar\Delta$ with
$\pi(y,\ell)=(\pi y,\ell)$ and we define $\bar\mu_\Delta=\pi_*\mu_\Delta$.
The map $\bar f$ plays the role of a ``nonuniformly expanding map''.

As in Subsection~\ref{sec:NUE}, we consider  the tails $\mu_Y(\tau>n)$.
In the exponential tail setting of~\cite{Young98}, $\mu_Y(\tau>n)=O(\gamma^n)$ for some $\gamma\in(0,1)$ and a version of the ``Sinai trick'' (see for example~\cite[Lem.~3.2]{MN05}) shows that
$v\circ\pi_\Delta=\hat v+\chi_1\circ f-\chi_1$ where
$\chi_1\in L^\infty$ and $\hat v(y)=\hat v(y')$ whenever $\pi y=\pi y'$.
In particular, $\hat v$ projects to a well-defined observable $\bar v:\bar\Delta\to\R^m$.

This construction can be carried out so that
$\bar v$ is sufficiently regular that
the analogue of
condition~\eqref{eq:P} holds, where $P$ is the transfer operator on $\bar\Delta$.
Hence we obtain an $L^\infty$ martingale-coboundary decomposition
$\bar v=m+\chi_2\circ \bar f-\chi_2$ on $\bar\Delta$.
This gives the associated decomposition
\[
v\circ\pi_\Delta=\hat v = m\circ\pi+\chi\circ f-\chi,
\]
on $\Delta$ where $\chi=\chi_1+\chi_2\circ\pi$.

Now the argument is finished, since we can apply the methods from
Subsection~\ref{sec:NUE} to obtain the WIP, etc, for $m$ on $(\bar\Delta,\bar\mu_\Delta)$, and hence $m\circ\pi$ on $(\Delta,\mu_\Delta)$, $\hat v$
on $(\Delta,\mu_\Delta)$, and $v$ on $(\Lambda,\mu)$.

Finally, we consider the case
$\mu(\tau>n)=O(n^{-(\beta+1)})$ with $\beta>1$.
In certain situations (nonuniform expansion but uniform contraction)
the Sinai trick works as above and reduces to the situation in~\eqref{eq:poly}.
The general case is more complicated but is covered by~\cite[Cor.~2.2]{MV16}.
Again, this is optimal since there are many examples with $\beta=1$ where the CLT with standard scaling does not hold.

\subsubsection{Hyperbolic flows} \label{sec:flow}

Let $\dot y=g(y)$ be an ODE defined by a $C^2$ vector field $g:M\to TM$ on a compact Riemannian manifold $M$.  Let $g_t:M\to M$ denote the corresponding flow.
Let $X \subset M $ be a codimension one cross-section transverse to the flow and let $\varphi:X\to\R^+$ be a return time function, namely a function such that $g_{\varphi(x)}(x)\in X$ for $x\in X$.  The map $T=g_\varphi:X\to X$ is called the {\em Poincar\'e map}.
We assume (possibly after shrinking $X$) that $\inf\varphi>0$.
Given an ergodic invariant probability measure $\mu_X$ on $X$ and $\varphi\in L^1(X)$, we construct an ergodic invariant probability measure $\mu$ on $M$ as follows.
Define the suspension
\[
X^\varphi=\{(x,u)\in X\times\R: 0\le u\le\varphi\}/\sim, \qquad
(x,\varphi(x))\sim (Tx,0).
\]
The suspension flow $T_t:X^\varphi\to X^\varphi$ is given by
$T_t(x,u)=(x,u+t)$ modulo identifications.
The probability measure $\mu^\varphi=(\mu_X\times{\rm Leb~})/\int_X\varphi\,d\mu_X$ is ergodic and $T_t$-invariant.
Moreover, $\pi:X^\varphi\to M$ given by $\pi(x,u)=T_ux$ is a semiconjugacy from
$T_t$ to $g_t$ and $\mu=\pi_*\mu^\varphi$ is the desired ergodic invariant probability measure on $M$.

Now suppose that $v\in C^\kappa_0(M,\R^m)$ and define the induced observable
\[
V:X\to\R^m, \qquad V(x)=\int_0^{\varphi(x)}v(g_u x)\,du.
\]
By a purely probabilistic argument~\cite{GM16} (based on~\cite{Ratner73,MT04,Gouezel07,MZ15}), the WIP for $V:X\to\R^m$ with the map $T$ implies a WIP for $v:M\to\R^m$.
That is, setting
$v_t=\int_0^t v\circ g_s\,ds$ and $W_n(t)=n^{-1/2}v_{nt}$, we obtain
$W_n\to_\mu W$ where $W$ is Brownian motion
 with
covariance $\Sigma=\Sigma_X/\int_X\varphi\,d\mu_X$ and
$\Sigma_X$ is the covariance in the WIP for $V$.

By Bowen~\cite{Bowen75}, Axiom~A flows can be realized as suspension flows over uniformly hyperbolic diffeomorphisms, and the above considerations yield the WIP for attractors for Axiom~A flows~\cite{DenkerPhilipp84}.  The same is true for large classes of nonuniformly hyperbolic flows modelled as suspensions over Young towers with summable decay of correlations, including Lorentz gases,
Lorenz attractors~\cite{AMV15} and singular hyperbolic attractors~\cite{AraujoMsub}.
In these situations, $\mu_X$ and $\mu$ are SRB measures on $X$ and $M$ respectively.  Also the nondegeneracy property in Proposition~\ref{prop:deg} applies to $\Sigma_X$ and thereby
$\Sigma=\Sigma_X/\int_X\varphi\,d\mu_X$.  Moreover,
\[
\Sigma=\lim_{t\to\infty}t^{-1}\int_M v_t\otimes v_t\,d\mu,
\]
and under extra (rather restrictive) mixing assumptions
\[
\Sigma=\int_0^\infty \int_M \{v\otimes (v\circ g_t)+(v\circ g_t)\otimes v\}\,d\mu\,dt.
\]

\section{General rough path theory}
\label{sec:RPs}

\subsection{Limit theorems from rough path analysis}\label{sec:RPs1}

Consider a (for simplicity only: finite-dimensional) Banach space $(\MB,\|\cdot \|)$ and fixed $p \in [2,3)$.
Define the group $G := \MB \oplus (\MB\otimes \MB )$
with multiplication $(a,M) \star (b,N) := (a + b, M + a \otimes b + N)$, inverse $(a,M)^{-1}:= (-a, -M + a\otimes a),$ and identity $(0,0)$.
A (level-$2$) {\it $p$-rough path} (over $\MB$, on $[0,1]$) is a path $\BX=(\BX_{t}:0 \le t \le 1)$ with values and increments $\BX_{s,t}:=\BX_s^{-1}\star \BX_t := (X_{s,t}, \X_{s,t}) \in G$ of finite $p$-variation condition, $p \in [2,3)$, either in the sense (``{\it homogeneous rough path norm}'')
\begin{equation} \label{equ:tripleenorm}
                   \vvvert \BX \vvvert_{p\var} := \| X \|_{p\var} +   \| \X \|^{1/2}_{(p/2)\var} < \infty \ ,
\end{equation}
or, equivalently, in terms of the {\it inhomogeneous rough path norm}
\begin{equation} \label{equ:doublenorm}
                   \| \BX \|_{p\var} := \| X \|_{p\var} +   \| \X \|_{(p/2)\var} < \infty \ ;
\end{equation}
we used the notation, applicable to any $\Xi$ from $\{ 0 \le s \le t \le 1 \}$ into a normed space, any $q>0$,
\begin{equation}  \label{equ:doubleRPnorm}
      \| \Xi \|_{q\var} := \left( \sup_{\op} \sum_{[s,t]\in \op} | \Xi_{s,t}|^{q}   \right)^{\frac 1 q} < \infty \ .
\end{equation}
Write $\CC^{p\var} ([0,1],\MB)$ resp.\ $\DD^{p\var} ([0,1],\MB)$ for the space of such (continuous resp.\ c{\`a}dl{\`a}g) $p$-rough paths;
and also $\CC$ resp.\ $\DD$ for the space of continuous resp.\ c{\`a}dl{\`a}g paths with values in $\MB \oplus (\MB\otimes \MB )$.
The space of $\alpha$-H{\"o}lder rough paths, $\CC^{\alpha\Hol}$ with $\alpha = 1/p \in (1/3,1/2]$ forms a popular subclass of $\CC^{p\var}$.
A weakly geometric $p$-rough path, in symbols $\BX \in \CC_g^{p\var}$, satisfies a ``product rule'' of the type $ \Sym (\X_t ) = (1/2) X_t \otimes X_t$; effectively $\BX$ takes values in a sub-group $H \subset G$.
(We remark that, when $\MB = \R^m$, the (Lie) groups $H,G,$ can be identified with, respectively, the step-$2$ truncated free nilpotent group with $m$ generators and the step-$2$ truncated Butcher group with $m$ decorations of its nodes.)
Every continuous BV path lifts canonically via $$\X_t = \int_0^t (X_s - X_0) \otimes dX_s \ ,$$ and gives rise to a (continuous) weakly geometric $p$-rough path.
Conversely, every $\BX \in \CC_g^{p\var}$ is the uniform limit of smooth paths, with uniform $p$-variation bounds.
Similarly, every c{\`a}dl{\`a}g BV path $X$ lifts canonically via $$\X_t = \int_{(0,t]} (X^-_s - X_0) \otimes dX_s \,$$ to a (c{\`a}dl{\`a}g) $p$-rough path in $\DD^{p\var}$. We introduce, on $\DD^{p\var}$ (and then by restriction on $\CC^{p\var}$ and $\CC_g^{p\var}$) the
 {\it (inhomogeneous) $p$-rough path distance}\footnote{In view of the genuine non-linearity of rough path spaces, we refrain from writing $\| \BX - \tilde \BX  \|_{p\var,[0,1]}$.}
\begin{equation} \label{equ:doublenormdist}
                   \| \BX ; \tilde \BX  \|_{p\var} := \| X - \tilde X \|_{p\var} +   \| \X -\tilde \X \|_{(p/2)\var}  \ .
\end{equation}
(A similar H\"older rough path distance can be defined on  $\CC^{\alpha\Hol}$ and $\CC_g^{\alpha\Hol}$).

\medskip
Consider sufficiently regular vector fields $V_0: \R^d \to \R^d$ and $V: \R^d \to L(\MB,\R^d)$.
By definition, $Y$ solves the {\it rough differential equation (RDE)}
$$
                dY = V_0(Y^-)dt + V(Y^-) d \BX
$$
if, for all $0\le s < t \le 1$, writing $DV$ for the derivative,\footnote{In coordinates, when $\MB = \R^m$, we have $DV (Y_s) V (Y_s) \X_{s,t} =  \partial_\alpha V_\gamma (Y_s) V^\alpha_\beta(Y_s) \X_{s,t}^{\beta,\gamma}$
with summation over $\alpha = 1, \ldots, d$ and $\beta, \gamma = 1, \ldots , m.$}
\begin{equation}                 \label{def:davieRDE}
               Y_t - Y_s = V_0(Y_s) (t-s) + V (Y_s) X_{s,t} + DV (Y_s) V (Y_s) \X_{s,t} + R_{s,t} \ ,
\end{equation}
where, for the ``remainder term'' $R$, we require,  writing $\op (\eps)$ for a partition of $[0,1]$ with mesh-size less than $\eps$,
$$ \sup_{\op (\eps)} \sum_{[s,t]\in \op (\eps)} | R_{s,t}|   \to 0  \ \ \text{ as $\eps \to 0$.} $$
This definition first encodes that $Y$ is controlled (cf. \cite{FS17}) by $X \in D^{p\var}$, with derivative $Y'=V(Y_s) \in D^{p\var}$ and remainder $Y_{s,t}^\# = W (Y_s) \X_{s,t} + R_{s,t}$ with $\| Y^\# \|_{(p/2)\var} < \infty$.
As a consequence, $Y$ satisfies a bona fide rough integral equation, for all $t \in (0,1]$,
$$
            Y_t = y_0 + \int_{(0,t]} V_0(Y^-_s) ds + \int_{(0,t]} V(Y^-_s) d \BX_s \ .
$$
Conversely,~\eqref{def:davieRDE} is satisfied by every solution to this integral equation. See e.g.~\cite{FH14} for more details on this construction in the H\"older rough path case, and~\cite{FS17, FZ17} for the c{\`a}dl{\`a}g $p$-variation case; this contains the discrete H{\"o}lder setting of~\cite{Kelly16}.
The following theorem, in the case of continuous $p$-rough paths, is due to Lyons \cite{Lyo98}, the recent extension to c{\`a}dl{\`a}g rough paths is taken from \cite{FZ17}. We write $C^{p+}$ to indicate $C^{p+\eps}$, for some $\eps>0$.

\begin{thm}[Continuity of RDE solution map]  \label{thm:ult}
Let $ p \in [2,3)$.
Consider a c{\`a}dl{\`a}g rough path $ \BX \in \DD^{p\var} ([0,1],\cB)$, and assume $V_0 \in C^{1+}$ and $V \in C^{p+}$.
Then there exists a unique c\`adl\`ag solution $Y \in D([0,1],\R^d)$ to the rough differential equation
\[
dY_t = V_0(Y^-_t)dt + V(Y^-_t) d \BX_t,  \ \ Y_0 = y_0  \in \R^d\;,
\]
and the solution is
locally Lipschitz in the sense that
$$ \|Y-\tY\|_{p\var} \lesssim  \|\BX ;\tBX \|_{p\var} %+ \|V-\tilde{V} \|_{q,[0,1]}
+ |y_0-\ty_0| $$
with proportionality constant uniform over bounded classes of driving $p$-rough paths.
\end{thm}

The $p$-variation rough path distance can be replaced by a $p$-variation Skorokhod type rough path metric, which adds more flexibility when the limiting (rough) path has jumps, but we won't need this generality here. %Another (minor) extension concerns drift vector field (s) driven by, say, $\Gamma$ \in C^{1-var}$:
Checking $p$-variation rough path convergence can be done by interpolation: uniform convergence plus uniform $p'$-variation bounds, for some $p' < p$.

It is known that (c{\`a}dl{\`a}g) semimartingales give rise to
c{\`a}dl{\`a}g $p$-rough paths for any $p \in (2,3)$~\cite{CF17x}.
Solving the resulting random RDE provides exactly a (robust) solution theory for the corresponding SDE.
As a consequence, we have the following limit theorems of Stratonovich and It\^o type, which cannot be obtained by UCV/UT type argument familiar from stochastic analysis.
(Assumptions on $V_0, V$ are as above.)
The following theorem applies in particular to sequences of smooth processes, in which case Stratonovich SDEs are simply random ODEs.

\begin{thm}[Stratonovich-type limit theorem]\label{thm:strat ult}
Consider a sequence of continuous semimartingale drivers $(B^n)$ with Stratonovich lift ($\BB^{\circ,n}$), such that $\BB^{\circ,n}$ converges to $\BB = (B, \BBB^\circ + \Gamma)$, for some continuous BV process $\Gamma$, weakly (resp.\ in probability, a.s.) in the uniform topology with $\{\|\BB^{\circ, n}\|_{p\var}(\omega)\}$ tight, for $p \in (2,3)$.
($\Gamma$ is necessarily skew-symmetric.)
\begin{enumerate}[label=(\roman*)]
\item \label{point:strat 1} For any $p'>p$, it holds that $\BB^{\circ,n} \to \BB$ weakly (resp.\ in probability, a.s.) in the $p'$-variation rough path topology.
\item \label{point:strat 2} Assume, in the sense of Stratonovich SDEs,
\footnote{Often $B^n$ has continuous BV sample paths. Every such process is (trivially) a semimartingale (under its own filtration); the Stratonovich SDE interpretation is
the one consistent with the ODE interpretation, in the sense of a Riemann-Stieltjes integral equation.}
\begin{equation*} % \label{equ:approxSDE}
dY^n_t = V_0(Y^{n}_t)dt + V(Y^{n}_t) \circ dB^n_t
\end{equation*}
such that $Y^n_0 \equiv y^n_0 \to y_0$. Then the Stratonovich SDE solutions $Y^n$ converge weakly (resp.\ in probability, a.s.) to $Y$ in the uniform topology, where $Y_0=y_0$ and
\begin{equation*}
dY_t = V_0(Y_t)dt + DV (Y_t) V (Y_t) d \Gamma_t + V(Y_t) \circ dB_t\;.
\end{equation*}
Moreover, $\{ \left\Vert Y^{n}\right\Vert _{p\var}\left( \omega \right) :n\geq 1 \} $ is tight and one also has weak (resp.\ in probability, a.s.) convergence in the $p^{\prime }$-variation uniform metric for any $p^{\prime }>p.$
\end{enumerate}
\end{thm}

We now state an analogous It{\^o}-type result.
The next theorem in particular applies to sequences of piecewise constant, c{\`a}dl{\`a}g processes, in which case, It{\^o} SDEs are simply stochastic recursions.

\begin{thm}[It{\^o}-type limit theorem]\label{thm:ito ult}
Consider a sequence of c{\`a}dl{\`a}g semimartingale drivers $B^n$, with It{\^o} lift $\BB^n = (B^n, \BBB^n)$, such that $\BB^n$ converges to $\BB = (B, \BBB + \Gamma)$, for some c{\`a}dl{\`a}g BV process $\Gamma$, weakly (resp.\ in probability, a.s.) in the  uniform topology with $\{\|\BB^n\|_{p\var}(\omega)\}$ tight, for $p \in (2,3)$.
\begin{enumerate}[label=(\roman*)]
\item For any $p'>p$, it holds that $\BB^{\circ,n} \to \BB$ weakly (resp.\ in probability, a.s.) in the $p'$-variation rough path topology.
\item Assume, in the sense of It{\^o} SDEs,
\begin{equation*}
dY^n_t = V_0(Y^{n,-}_t)dt + V(Y^{n,-}_t) dB^n_t
\end{equation*}
such that $Y^n_0 \equiv y^n_0 \to y$. Then It{\^o} SDE solutions $Y^n$ converge weakly (resp.\ in probability, a.s.) to $Y$ in the uniform topology, where $Y(0)=y$ and
\begin{equation*}
dY_t = V_0(Y^-_t)dt + DV (Y^-_t) V (Y^-_t) d \Gamma_t + W(Y^-_t) dB_t \;.
\end{equation*}
Moreover, $\{ \left\Vert Y^{n}\right\Vert _{p\var}\left( \omega \right) :n\geq 1 \} $ is tight and one also has weak (resp.\ in probability, a.s.) convergence in the $p^{\prime }$-variation uniform metric for any $p^{\prime }>p.$
\end{enumerate}
\end{thm}

\begin{rem}
A minor generalization of Theorem \ref{thm:ito ult}, which will be convenient later on, states that the drift term $V_0(Y^{n,-}_t)dt$ in the approximate problem can be replaced by $V_0(Y^{n,-}_t)d\tau^n$ where $\tau^n(t) \to t$ uniformly with uniform $1$-variation bounds.
\end{rem}

We emphasize that in both Theorem~\ref{thm:strat ult} and~\ref{thm:ito ult}, the {\it only purpose} of the semimartingale and adaptedness assumptions is to give a familiar interpretation of what is  really a rough differential equation driven by a random rough path. (Consistency with SDEs, in a general semimartingale setting, is established in~\cite{CF17x}).
The proof is essentially a corollary of interpolation, in a weak convergence setting, with the purely deterministic Theorem~\ref{thm:ult},
see~\cite{FZ17} for details.

\subsection{WIPs in rough path theory}

We start with some generalities. A {\it Brownian rough path} (over $\R^m$) is an $\R^m \oplus (\R^m \otimes \R^m)$-valued continuous process $\BB = (B, \BBB)$ with independent increments with respect to the group structure introduced in Section \ref{sec:RPs1}, such that $B$ is centered. (In particular, $B$ is a classical $m$-dimensional Brownian motion.) It is known that sample paths $\BB(\omega)$ are, with probability one, in $\CC^{\alpha\Hol}$ for any $\alpha <1/2$, and hence also in $\CC^{p\var}$ and $\DD^{p\var}$ for any $p>2$. We have a full characterization of Brownian rough paths: $B$ is a classical $m$-dimensional Brownian motion (with some covariance $\Sigma \in \R^m \otimes \R^m$) and
$$
      \BBB_{s,t} = \int_s^t B_{s,r} \otimes \circ dB_r + (t-s) \Gamma  =  \int_s^t B_{s,r} \otimes dB_r + (t-s) (\Gamma + \tfrac{1}{2} \Sigma)
     $$
for some matrix $\Gamma \in \R^m \otimes \R^m$, which we name {\it area drift}. (Note that $\BB$ is geometric iff $\Gamma $ is skew-symmetric.)
Given a sequence of random rough paths $(\BB_\eps)$, we say that the WIP holds {\it in $\alpha$-H\"older} (resp.\ {\it $p$-variation}) {\it rough path sense} if, as $\eps \to 0$,
\begin{equation}  \label{equ:WIPinCC}
          \BB_{\eps}  \to_{w} \BB  \text{ in $\CC^{\alpha\Hol}$ (resp.\ $\CC^{p\var}$, \ $\DD^{p\var}$)}
\end{equation}
but note that only the regimes $\alpha \in (1/3,1/2)$ (resp.\ $p \in (2,3)$) correspond to a WIP in a bona fide rough path topology.
As is implicit in Theorems \ref{thm:strat ult} and \ref{thm:ito ult}, this follows from checking convergence in law in the uniform topology; that is, an {\it enhanced weak invariance principle} in the sense that
\footnote{Again it suffices to work with the uniform topology on
 both $\CC$ and $\DD$.}
$$
          \BB_{\eps}  \to_{w} \BB  \text{ in $\CC$ (resp.\ $\DD$) as $\eps \to 0$}\;,
$$
together with tightness of $\alpha$-H\"older (resp.\ $p$-variation) rough path norms, at the expense of replacing $\alpha$ (resp.\ $p$) in
(\ref{equ:WIPinCC}) with $\alpha' < \alpha$ (resp.\ $p' > p$).

\medskip
In simple situations, $\BB_{\eps}$ is given as the canonical (Stratonovich or It{\^o}) lift
of a {\it good} sequence of convergent semimartingales. In this case, the limiting area drift is zero and $ \{ \vvvert \BB_{\eps} \vvvert_{p\var} : \eps \in (0,1]\} $ is automatically tight \cite{CF17x}, for any $p>2$.
(This gives a decisive link between classical semimartingale stability theory \cite{JMP89, KP91} with rough path analysis.) Immediate applications then include
Donsker's theorem in $p$-variation rough path topology, under identical (finite second) moment assumptions as the classical Donsker theorem. (With piecewise linear interpolation, the limit is the Stratonovich lift $(B, \int B \otimes \circ dB)$, with piecewise constant interpolation the limit is the It{\^o} lift $(B, \int B \otimes dB)$.)  As another immediate application,
the (functional) CLT for $L^2$-martingales with stationary, ergodic increments is valid on a rough path level \cite{CF17x}. If interested in the $\alpha$-H\"older rough path topology,
one can use a Kolmogorov-type tightness criterion \cite{FV10}. Provided $q>1$, a uniform moment estimate of the form
\begin{equation}   \label{equ:UME}
\sup_{\eps \in (0,1]}  \E\Big[ \vvvert \BB_\eps(s,t) \vvvert^{2q}\Big]^{1/2q}
\lesssim |t-s|^{1/2}\;,
\end{equation}
or equivalently,
\[
\sup_{\eps \in (0,1]}  \E\Big[ | B_\eps(s,t) |^{2q}\Big]^{1/2q}
\lesssim |t-s|^{1/2}\;, \ \ \  \sup_{\eps \in (0,1]}  \E\Big[ | \BBB_\eps(s,t) |^{q}\Big]^{1/q}
\lesssim |t-s| \ ,
\]
gives tightness in the $\alpha$-H\"older topology, for every $\alpha < 1/2 - 1/(2q)$.
\begin{rem}\label{rem:holder tight}
In the H{\"o}lder setting, note that only $\alpha > \frac{1}{3}$ gives a bona fide (level-$2$) rough path metric under which the It\^o map behaves continuously. This leads to the suboptimal moment assumption $q>3$.
To obtain the WIP in the H\"older rough path sense \cite{BFH09}, this necessitates increments with $6+$ moments.
(In contrast, we have the WIP in the $p$-variation rough path topology under the optimal assumption of $2+$ moments.)
This is also the main drawback of using the H{\"o}lder topology in~\cite{KM16}.

Of course, in Gaussian situations such moment assumptions are harmless and can conveniently be reduced to $q=1$. An instructive example is given by physical Brownian motion $X^\eps$, as introduced in Section \ref{sec:WIP}. The tightness condition can be seen to be satisfied for all $q< \infty$, giving $\alpha$-H\"older rough path tightness for any $\alpha < 1/2$. More interestingly, $X^\eps$ has a Brownian rough path limit with non-zero area drift \cite{FGL15}, provided the particle feels a Lorentz force, expressed through  non-symmetry of $M$. (See the notation in Section \ref{sec:WIP}.) Specifically, this is seen by writing $M X^\eps$ as Brownian motion plus a ``corrector'' which goes uniformly to zero, but leads, in the $\eps \to 0$ limit, to an area contribution.
\end{rem}

\begin{rem}
We note that (\ref{equ:UME}) requires no martingale assumptions whatsoever. It is an important observation for the sequel that (\ref{equ:UME}) leads to tightness not only  in the $\alpha$-H\"older rough path topology but also in the $p$-variation rough path topology: to wit, it follows from the Besov-variation embedding \cite{FV06} that (\ref{equ:UME}) implies $p$-variation tightness for any $p>2$. In this way, for example, one can reprove the WIP in $p$-variation rough path topology under the almost optimal assumption of $2+$ moments. This argument becomes important when direct martingale arguments are not possible.
\end{rem}

\section{Applications to fast-slow systems}
\label{sec:appl}

 \subsection{Chaotic dynamics: enhanced WIP and moments}
\label{sec:chaos2}

In this subsection, we resume the discussion of chaotic dynamical systems from Section~\ref{sec:chaos} but now focusing on some finer statistical properties, namely an enhanced WIP and moment estimates, that are required for applying rough path theory.

\subsubsection{Expanding maps} \label{sec:NUE2}

Continuing Subsection~\ref{sec:NUE}, we suppose that $T:M\to M$ is a $C^2$ uniformly expanding map, with unique ergodic absolutely continuous invariant probability measure $\mu$, so
conditions~\eqref{eq:decay} and~\eqref{eq:P} hold.
In particular, for any $v\in C^\kappa_0(M,\R^m)$, we have an $L^2$
martingale-coboundary decomposition decomposition $v=m+\chi\circ T-\chi$.
Define the c\`adl\`ag processes $W_n\in D([0,1],\R^m)$,
$\BBW_n\in D([0,1],\R^m \otimes \R^m)$,
\[
W_n(t)=n^{-1/2}\sum_{0\le j<n}v\circ T^j, \qquad
\BBW_n(t)=n^{-1}\sum_{0\le i< j<n}(v\circ T^i)\otimes(v\circ T^j).
\]
Recall that we have the WIP $W_n\to_\mu W$ in $D([0,1],\R^m)$ where $W$ is $m$-dimensional Brownian motion with covariance $\Sigma$ given by formulas~\eqref{eq:Sigma1} and~\eqref{eq:Sigma2}.
By~\cite[Thm.~4.3]{KM16}, we have the {\em enhanced WIP} (called iterated WIP in~\cite{KM16})
\[
(W_n,\BBW_n)\to_\mu (W,\BBW) \quad\text{in $\DD([0,1],\R^m\times(\R^m \otimes \R^m))$},
\]
where $W(t)=\int_0^t W\otimes dW+\Gamma t$.
Here $\int W\otimes dW$ is the It\^o integral and the area drift $\Gamma \in \R^m \otimes \R^m$ is given by
\begin{align} \label{eq:Gamma1}
\Gamma =\lim_{n\to\infty} \int_M \BBW_n(1)\,d\mu,
\end{align}
and satisfies
\begin{align} \label{eq:Gamma2}
\Gamma =\sum_{n=1}^\infty \int_M v\otimes (v\circ T^n)\,d\mu.
\end{align}
The proof of the enhanced WIP in~\cite{KM16} has two main steps.  The first step is to apply~\cite[Thm.~2.2]{KP91} (alternatively~\cite{JMP89}) to the martingale component $m$ taking into consideration that $\{m\circ T^n,\,n\ge0\}$ is a reverse martingale difference sequence.  This yields an enhanced WIP with zero area drift.  The contribution from the coboundary $\chi\circ T-\chi$ is no longer negligible, but a general result~\cite[Thm.~3.1]{KM16} for mixing dynamical systems and $L^2$ coboundaries yields the (typically nonzero) area drift $\Gamma$.

Again, strong distributional convergence applies by~\cite[Thm.~1]{Zweimuller07}.  The hypotheses in~\cite{Zweimuller07} are verified in the course of the proof of~\cite[Lem.~6.3]{KM16}.  Hence
$(W_n,\BBW_n)\to_\lambda (W,\BBW)$ in $\DD([0,1],\R^m\times (\R^m \otimes \R^m))$
for all absolutely continuous Borel probability measures $\lambda$.

For nonuniformly expanding maps, the enhanced WIP goes through unchanged provided the martingale-coboundary decomposition holds in $L^2$ (with the usual caveat that
formula~\eqref{eq:Gamma2} only holds when $\mu$ is mixing).
  This covers the situation~\eqref{eq:poly} with $\beta>2$.
As before, extra work is required for the case $\beta\in(1,2]$.
By~\cite[Thm.~10.2]{KM16}, the enhanced WIP holds for all $\beta>1$ for nonuniformly expanding maps modelled by Young towers, including the intermittent maps~\eqref{eq:LSV}.   For such maps, we obtain optimal results: the enhanced WIP holds precisely when the ordinary CLT holds.

Turning to moments, an immediate consequence of the $L^p$ martingale-coboundary decomposition, $p\ge2$, and Burkh\"older's inequality is that
$\|v_n\|_p=O(n^{1/2})$ where the implied constant depends on $v$ and $p$.
As noted in~\cite{MN08,M09b,MTorok12},
in fact
\begin{equation} \label{eq:v_n max}
  \|v_n\|_{2p}=O(n^{1/2})
\end{equation}
and this holds for $L^p$ martingale-coboundary decompositions with $p\ge1$.
This improved result uses the additional information
that $v\in L^\infty$ and a maximal inequality of~\cite{Rio00}.

In the situation~\eqref{eq:poly}, we have the martingale-coboundary decomposition for all $1 \le p < \beta$.
As shown in~\cite{MN08,M09b}, the estimate~\eqref{eq:v_n max} is sharp; $\|v_n\|_q=O(n^{1/2})$ for $q<2\beta$ but there are examples where the estimate typically fails for $q>2\beta$.

We also require estimates for the enhanced (iterated) moment
$S_n=\sum_{0\le i\le j<n}(v\circ T^i)\otimes(v\circ T^j)$.
Assuming an $L^p$ martingale-coboundary decomposition with $p\ge3$ and $v\in L^\infty$, it was shown in~\cite[Prop.~7.1]{KM16} that
$\|S_n\|_{2p/3}=O(n)$.  In the Young tower setting, this has been improved in~\cite{KKMprep} to $\|S_n\|_p=O(n)$ for $p\ge1$.

The moment estimates discussed above are all in $L^p$ spaces with respect to $\mu$.  Clearly if $\lambda\ll\mu$ and $d\lambda/d\mu\in L^\infty$ then the same moment estimates hold also with respect to $\lambda$.  In particular, we can take $\lambda=\Vol$ for the $C^2$ uniformly expanding maps.
For the intermittent maps~\eqref{eq:LSV}  it is standard that $d\mu/d\Leb$ is bounded below, so we can take $\lambda=\Leb$.

\subsubsection{Hyperbolic diffeomorphism} \label{sec:NUH2}

For Axiom~A diffeomorphisms and Young towers with exponential tails, we saw
in Subsection~\ref{sec:NUH} that there is an $L^p$ martingale-coboundary decomposition for all $p$.  Also, for Young towers with exponential contraction and polynomial tails $\mu_Y(\tau>n)=O(n^{-(\beta+1)})$, we have an $L^p$ martingale-coboundary decomposition for $p<\beta$.   By~\cite{KM16,KKMprep}, we obtain the enhanced WIP provided $\beta>1$ and optimal moment estimates
$\|v_n\|_{2p}=O(n^{1/2})$ and $\|S_n\|_{p}=O(n)$ for $1\le p<\beta$.
As before, the covariance $\Sigma$ and drift $\Gamma$ satisfy~\eqref{eq:Sigma1} and~\eqref{eq:Gamma1}, and under additional mixing assumptions we have~\eqref{eq:Sigma2} and~\eqref{eq:Gamma2}.

For general Young towers with polynomial tails $\mu_Y(\tau>n)=O(n^{-\beta})$
the enhanced WIP still holds for all $\beta>1$ by~\cite{MV16}
but currently we only have the moment estimates
$\|v_n\|_{2p}=O(n^{1/2})$ and $\|S_n\|_{2p/3}=O(n)$ for $3\le p<\beta$
from~\cite{KM16}.  Obtaining optimal moment estimates here is the subject of work in progress.

\subsubsection{Hyperbolic flows} \label{sec:flow2}

The methods mentioned in Subsection~\ref{sec:flow} for passing the WIP from
(non)uniformly hyperbolic diffeomorphisms to
(non)uniformly hyperbolic flows work just as well for the enhanced WIP~\cite[Sec.~6]{KM16}.
Define
\[
W_n(t)=n^{-1/2}\int_0^{nt} v\circ g_s\,ds\;, \qquad
\BBW_n(t)=n^{-1}\int_0^{nt}\int_0^s (v\circ g_r)\otimes(v\circ g_s)\,dr\,ds\;.
\]
Then $(W_n,\BBW_n)\to_\lambda (W,\BBW)$ where $W$ is Brownian motion with covariance $\Sigma$ and
$\BBW(t)=\int_0^t W\otimes dW+\Gamma_I t$.
Here $\Sigma=\lim_{n\to\infty}\E_\lambda( W_n(1)\otimes W_n(1))$ as before,  and
$\Gamma_I=\lim_{n\to\infty}\E_\lambda \BBW_n(1)$.
Alternatively,
$\BBW(t)=\int_0^t W\otimes\circ dW+\Gamma t$,
where $\Gamma=\Gamma_I-\frac12\Sigma$ is skew-symmetric.
Under extra mixing assumptions,
\begin{align} \label{eq:Sigflow}
  \Sigma &= \int_0^\infty \int_\Lambda \{v\otimes (v\circ g_t)+(v\circ g_t)\otimes v\}\,d\mu\,dt\;,
  \\
\label{eq:Gamflow}
  \Gamma &= \frac12\int_0^\infty \int_\Lambda \{v\otimes (v\circ g_t)-(v\circ g_t)\otimes v\}\,d\mu\,dt\;.
\end{align}

The situation for moments extends in a straightforward way~\cite[Sec.~7.2]{KM16}.
Define
\[
v_n=\int_0^{n} v\circ g_s\,ds\;, \qquad
S_n=\int_0^{n}\int_0^s (v\circ g_r)\otimes(v\circ g_s)\,dr\,ds\;.
\]
Then the estimates described in Subsection~\ref{sec:NUE} apply equally here.

 \subsection{Continuous dynamics}
\label{subsec:cont dynamics}
We present now an application of rough path theory to fast-slow systems where the fast variable satisfies a suitable WIP.
We consider first continuous dynamics~\eqref{eq:ODE} in the case of multiplicative noise, i.e.,
\begin{align} \label{eq:prod}
\dot x_\eps &=a(x_\eps)+\eps^{-1}b(x_\eps)v(y_\eps)\;, \qquad
\dot y_\eps =\eps^{-2}g(y_\eps) \; .
\end{align}
We consider the $\R^m$-valued path
\begin{align*}
 \Wve(t)=\eps\int_0^{t\eps^{-2}}v \circ g_{s}\,ds
\end{align*}
and rewrite the slow dynamics in the form of a controlled ODE,
$$
     dx_\varepsilon = a(x_\varepsilon) dt + b(x_\eps) d W_\eps \ .
$$
Following Kelly--Melbourne~\cite{KM16}, this formulation invites an application of finite-dimensional rough path theory; the only modification relative to~\cite{KM16} is our present use of
$p$-variation rough path metrics, which leads to optimal moment assumptions and optimal regularity assumptions on the coefficients. The key is a suitable WIP on the level of rough paths, as discussed in Section~\ref{sec:RPs}.

To this end, we consider the following two assumptions on the fast dynamics.
Following the discussion in Section~\ref{sec:chaos2}, we see that a wide range of dynamics satisfy these assumptions.
For every $\eps > 0$, we let $\BBWve$ be the canonical second iterated integral of $\Wve$, and for $p \in (2,3)$, we consider the geometric $p$-rough path $\BWve := (\Wve,\BBWve)$.

\begin{assumption}\label{ass:enhanced WIP cont}
It holds that $(\Wve,\BBWve)\to(W,\BBW)$ as $\eps\to0$
in the sense of finite-dimensional distributions on $(M,\lambda)$,
where $W$ is an $m$-dimensional Brownian motion and $\BBW(t)=\int_0^t W\otimes \circ dW+\Gamma t$
for some $\Gamma\in\R^m \otimes \R^m$ deterministic.
\end{assumption}

\begin{assumption} \label{ass:momentproductcase} There exists $q>1$ and $K>0$ such that
\begin{align*}
&\Big\|\int_s^t v^i\circ g_r\,dr\Big\|_{L^{2q}(\lambda)}\le K|t-s|^{1/2}\;,
\\
&\Big\|\int_s^t\int_s^r v^i\circ g_u\,v^j\circ g_r\,du\,dr\Big\|_{L^{q}(\lambda)}\le K|t-s|\;,
\end{align*}
for all $s,t\ge0$ and $1\le i,j\le m$.
\end{assumption}

The first assumption identifies the possible limit points of $\Wve$ as a rough path; the second ensures that the WIP holds in a sufficiently strong  rough path topology as demonstrated by the following result.
As before, all path space norms ($p\var$, $\alpha\Hol$, etc.) are relative to the fixed interval $[0,1]$.

\begin{prop}\label{prop:pvar tight}
Under Assumption~\ref{ass:momentproductcase}, it holds that for all $p \in (2,3)$
\[
\sup_{\eps \in (0,1]}\E \vvvert\BWve\vvvert^{2q}_{p\var} < \infty\;,
\]
and for all $\alpha \in (0,\frac{1}{2}-\frac{1}{2q})$
\[
\sup_{\eps \in (0,1]}\E \vvvert\BWve\vvvert^{2q}_{\alpha\Hol} < \infty\;.
\]
\end{prop}

\begin{proof}
Viewing $\BWve$ as a path in $G^2(\R^m) \subset \R^m \oplus (\R^m \otimes \R^m)$, the step-$2$ free nilpotent group equipped with Carnot-Carath{\'e}odory metric $d$, it holds that
\[
|d(\BWve(s),\BWve(t))|_{L^{2q}(\lambda)} \lesssim |\Wve(s,t)|_{L^{2q}(\lambda)} + |\BBWve(s,t)|^{1/2}_{L^{q}(\lambda)} \lesssim |t-s|^{1/2},
\]
where the final bound follows from Assumption~\ref{ass:momentproductcase}.
Let $\beta \in [0,1/2)$. Then $\E[|\BWve|^{2q}_{W^{\beta,2q}}]$ is uniformly bounded in $\eps > 0$, and thus, by the Besov-H{\"o}lder and Besov-variation embeddings~\cite{FV06} (see also~\cite[Cor.~A.2,~A.3]{FV10}), so is $\E\vvvert\BWve\vvvert_{(\beta-1/(2q))\Hol}^q$ and $\E\vvvert\BWve\vvvert_{(1/\beta)\var}^{2q}$.
\end{proof}

\begin{thm}\label{thm:prod}
Suppose Assumptions~\ref{ass:enhanced WIP cont} and~\ref{ass:momentproductcase} hold.
\begin{enumerate}[label=(\roman*)]
\item \label{point:prod 1} For every $p > 2$, it holds that $\BWve \to_\lambda \BW$ in the $p$-variation rough path topology.
\item \label{point:prod 2} Let $a\in C^{1+}(\R^d,\R^d)$, $b\in C^{2+}(\R^d,\R^{d\times m})$, and let $x_\eps$ be the solution to~\eqref{eq:prod}.
Then $x_\eps\to_{\lambda} X$ in $C^{p\var}([0,1],\R^d)$ for every $p>2$, where $X$ is the solution to the SDE
\begin{equation}\label{eq:prod SDE}
dX=\Big(a(X)+\sum_{i,j=1}^m \Gamma^{i,j}\sum_{k=1}^d b^{i,k}\partial_k b^j(X)\Big)dt+ b(X) \circ \, dW, \quad X(0)=\xi.
\end{equation}
\end{enumerate}
\end{thm}

\begin{rem} If follows from Assumptions~\ref{ass:enhanced WIP cont} and \ref{ass:momentproductcase} that the covariance matrix $\Sigma$ and area drift $\Gamma$ are given by
\[
\Sigma = \lim_{\eps \to 0} \E_\lambda ( \Wve(1) \otimes \Wve(1) ),
\qquad
\Gamma = \lim_{\eps \to 0} \E_\lambda \BBWve(1) - \frac{1}{2} \Sigma.
\]
Under additional mixing assumptions, formulas~\eqref{eq:Sigflow} and~\eqref{eq:Gamflow} hold.
\end{rem}

\begin{proof}
\ref{point:prod 1} follows from part~\ref{point:strat 1} of Theorem~\ref{thm:strat ult}.
For~\ref{point:prod 2}, observe that $x_\eps$ solves the ODE
\[
d x_\eps = a(x_\eps)dt + b(x_\eps)d\Wve\;.
\]
We are thus in the framework of part~\ref{point:strat 2} of Theorem~\ref{thm:strat ult}, from which the conclusion follows.
\end{proof}

 \subsection{Discrete dynamics}
 \label{subsec:discrete dynamics}
We now discuss discrete dynamics~\eqref{equ:discreteFastSlowGeneral} in the case of multiplicative noise, i.e.,
\begin{equation*}
X^{(n)}_{j+1} = X^{(n)}_j + n^{-1}a(X_j^{(n)}) + n^{-1/2}b(X_j^{(n)})v(Y_j)\;,
\end{equation*}
where, as before, $v : M \to \R^m$, $b : \R^d \to \R^{d \times m}$, and $a : \R^d \to \R^d$.
As usual, $X^{(n)}_0 = \xi \in \R^d$ is fixed and $Y_0$ is drawn randomly from a probability measure $\lambda$ on $M$.
To consider this system as a controlled ODE, we introduced the c{\`a}dl{\`a}g path
\begin{equation}\label{eq:x_n discrete}
x_{n} : [0,1] \to \R^d\;, \quad x_{n}(t) = X^{(n)}_{\floor{nt}}\; ,
\end{equation}
as well as the the c{\`a}dl{\`a}g paths
\begin{align*}
\Wvn &: [0,1] \to \R^m\;, \quad \Wvn(t) = n^{-1/2}\sum_{j=0}^{\floor{nt}-1} v(Y_j) \;,
\\
z_n &: [0,1] \to \R\;, \quad z_n(t) = \floor{tn}/n \;.
\end{align*}
It is easy to verify that $x_n$ defined by~\eqref{eq:x_n discrete} is the unique solution of the controlled (discontinuous) ODE
\begin{equation}\label{eq:prodDiscrete}
dx_n = a(x_n^-)dz_n + b(x_n^-) dW_n\;, \quad x_n(0) = \xi \in \R^d\;.
\end{equation}
Let us denote by $\BBWvn$ the canonical second iterated integral of $\Wvn$
\[
\BBW^{i,j}_{v,n}(s,t) = \int_{(s,t]}  (W^{i,-}_{v,n} (r) - W^i_{v,n} (s)) dW^j_{v,n}(r)\;, \quad 1 \leq i,j \leq m\;.
\]
Consider the following analogues of Assumptions~\ref{ass:enhanced WIP cont} and~\ref{ass:momentproductcase}.

\begin{assumption}\label{ass:enhanced WIP disc}
It holds that $(\Wvn, \BBWvn) \to (W,\BBW)$ as $\eps\to0$
in the sense of finite-dimensional distributions on $(M,\lambda)$,
where $W$ is a Brownian motion in $\R^m$ and $\BBW(t) = \int_0^t W\otimes dW + \Gamma t$ for some $\Gamma\in \R^{m\times m}$ deterministic.
\end{assumption}

\begin{assumption} \label{ass:momentDiscreteProd} There exists $q>1$ and $K>0$ such that for all $n \geq 1$ and $0 \leq k,l \leq n$,
\begin{align*}
&\big\|\Wvn(l/n) - \Wvn(k/n)\big\|_{L^{2q}(\lambda)}\le Kn^{-1/2}|l-k|^{1/2}\;,
\\
&\big\|\BBWvn(k/n,l/n)\big\|_{L^{q}(\lambda)}\le Kn^{-1}|l-k|\;.
\end{align*}
\end{assumption}

\begin{prop}\label{prop:pvar tight disc}
Under Assumption~\ref{ass:momentDiscreteProd}, for all $p \in (2,3)$
\[
\sup_{n \geq 1}\E \vvvert (\Wvn, \BBWvn) \vvvert^{2q}_{p\var} < \infty\;.
\]
\end{prop}

\begin{proof}
This is a direct application of~\cite[Prop.~6.17]{FZ17}.
\end{proof}

\begin{rem} If follows from Assumptions~\ref{ass:enhanced WIP disc} and \ref{ass:momentDiscreteProd} that the covariance matrix $\Sigma$ and the area drift $\Gamma$ are given by
\[
\Sigma = \lim_{n \to \infty} \E_\lambda ( \Wvn(1) \otimes \Wvn(1) ),
\qquad
\Gamma = \lim_{n \to \infty} \E_\lambda \BBWvn(1).
\]
Under additional mixing assumptions, formulas~\eqref{eq:Sigma2} and~\eqref{eq:Gamma2} hold.
\end{rem}

Combining Theorem~\ref{thm:ito ult} and Proposition~\ref{prop:pvar tight disc}, we arrive at the following convergence result which relaxes the moment conditions required in~\cite{KM16}.

\begin{thm}
Suppose that Assumptions~\ref{ass:enhanced WIP disc} and~\ref{ass:momentDiscreteProd} hold.
\begin{enumerate}[label=(\roman*)]
\item \label{point:disc 1} For every $p > 2$, it holds that $\BWvn \to_\lambda \BW$ in the $p$-variation rough path topology.
\item \label{point:disc 2} Let $a\in C^{1+}(\R^d,\R^d)$, $b\in C^{2+}(\R^d,\R^{d\times m})$, and let $x_n$ be the solution to~\eqref{eq:prodDiscrete}.
Then $x_n\to_\lambda X$ in $C^{p\var}([0,1],\R^d)$ for all $p>2$, where $X$ is the solution to the SDE
\begin{equation*}\label{eq:prod SDE discrete}
dX=\Big(a(X)+\sum_{i,j=1}^m \Gamma^{i,j}\sum_{k=1}^d b^{i,k}\partial_k b^j(X)\Big)dt+ b(X) \, dW\;, \quad X(0)=\xi\;.
\end{equation*}
\end{enumerate}
\end{thm}

\section{Extension to families and non-product case}\label{sec:exts}

Throughout this article, we restricted attention to the case of multiplicative noise given in product form.
The general form~\eqref{eq:ODE} was addressed in~\cite{KM17}, though with suboptimal moment assumptions.
By adapting the methods of this article to an infinite-dimensional rough paths setting similar to~\cite{KM17}, we are able to handle, with optimal moment assumptions, a generalisation of~\eqref{eq:ODE} of the form
\begin{align} \label{eq:ODE families}
\frac{d}{dt} x_\eps &=a_\eps(x_\eps,y_\eps) +\eps^{-1}b_\eps(x_\eps,y_\eps), \qquad
\frac{d}{dt} y_\eps =\eps^{-2}g_\eps(y_\eps)\;,
\end{align}
where $a_\eps,b_\eps,g_\eps$ now depend on $\eps$, and so does the probability measure $\lambda_\eps$ from which $y_\eps(0)$ is drawn randomly.
We assume we are also given a family $\mu_\eps$ of ergodic $g_{\eps,t}$-invariant probability measures on $M$, where $g_{\eps,t}$ is the flow generated by $g_\eps$; we require that $\int_M b_\eps(x,y)\,d\mu_\eps(y)=0$ for all $\eps\in[0,1]$ and $x\in\R^d$.

We note that a similar generalisation is also possible for the discrete dynamics~\eqref{equ:discreteFastSlowGeneral}, which was not addressed in~\cite{KM17} even in the $\eps$-independent setting.
Details are found in our forthcoming work~\cite{CFKMZPrep}.

Let $C_\eps^\eta(M,\R^m)$ be the space of $C^\eta$ functions
$v:M\to\R^m$ with $\int_M v\,d\mu_\eps=0$.
Fix $q \in (1,\infty]$, $\kappa, \bar\kappa > 0$, $\alpha > 2+\frac{d}{q}$.
Let $a_\eps \in C^{1+\bar\kappa,0}(\R^d\times M,\R^d)$ and $b_\eps \in C^{\alpha,\kappa}_\eps(\R^d\times M,\R^d)$ satisfying
\[
\sup_{\eps \in [0,1]} \|a_\eps\|_{C^{1+\bar\kappa,0}} < \infty\;,
\quad \sup_{\eps \in [0,1]} \|b_\eps\|_{C^{\alpha,\kappa}} < \infty\;,
\quad \lim_{\eps\to0}\|b_\eps-b_0\|_{C^{\alpha,\kappa}}=0\;.
\]
For $v \in C_\eps^\eta(M,\R^m)$, define
\[
  W_{v,\eps} (t) = \eps \int_0^{\eps^{-2}t} v \circ g_{\eps,s} \, ds,
  \qquad
  \BBW_{v, \eps} (t) = \int_0^t W_{v,\eps} \otimes dW_{v, \eps}
  .
\]
We require the following assumptions.
\begin{enumerate}
\item Moment bounds:
there exists $K > 0$ such that for all families $v_\eps,\,w_\eps\in C_\eps^\kappa(M)$,
it holds that
for all $s,t\ge0$ and $\eps \in [0,1]$,
\begin{align*}
&\Big\|\int_s^t v_\eps\circ g_{\eps,r}\,dr\Big\|_{L^{2q}(\lambda_\eps)}\le K\|v_\eps\|_{C^\kappa}|t-s|^{1/2}\;,
\\
&\Big\|\int_s^t\int_s^r v_\eps\circ g_{\eps,u}\,w_\eps\circ g_{\eps,r}\,du\,dr\Big\|_{L^{q}(\lambda_\eps)}\le K \|v_\eps\|_{C^\kappa}\|w_\eps\|_{C^\kappa} |t-s|\;.
\end{align*}
\item Enhanced WIP: there exists a bilinear operator $\mfB:C_0^\eta(M)\times C_0^\eta(M) \to\R$ such that for every family $v_\eps\in C_\eps^\kappa(M,\R^m)$ with $\lim_{\eps\to0}|v_\eps-v_0|_{C^\kappa}=0$,
there exists an $m$-dimensional Brownian motion $W$ such that
\[
(W_{v_\eps,\eps},\BBW_{v_\eps,\eps})\to_{\lambda_\eps} (W,\BBW), \quad\text{as $\eps\to0$},
\]
in the sense of finite-dimensional distributions, where
$\BBW^{i,j}(t)=\int_0^t W^i\, dW^j+\mfB(v_0^i,v_0^j)t$.
\item Convergence of drift: it holds that
\[
  \sup_{t\in[0,1]} \|V_\eps(t)-\bar a t\|_{C^{1+\bar\kappa}} \to_{\lambda_\eps} 0
  \quad \text{as }\eps \to 0,
\]
where $V_\eps(t) = \int_0^t a_\eps(\cdot,y_\eps(r))dr$
and $\bar a = \int_Ma_0(\cdot,y)d\mu_0(y)$.
\end{enumerate}

Consider the SDE
\begin{align}\label{eq:nonprod sde}
dX = \tilde a(X)\,dt + \sigma(X)\,dB\;, \quad X(0) = \xi\;,
\end{align}
where
$B$ is the standard Brownian motion in $\R^d$ and $\tilde a$ and $\sigma$ are given by
\begin{align*}
\tilde a^i(x) = \bar a^i(x) +  \sum_{k=1}^d \mfB(b_0^k(x,\cdot),\partial_k b_0^i(x,\cdot))\;,
\quad i=1,\dots,d\;,
\end{align*}
\begin{align*}
(\sigma (x) \sigma^T(x))^{ij} = \mfB(b^i_0(x,\cdot),b_0^j(x,\cdot)) + \mfB(b_0^j(x,\cdot),b^i_0(x,\cdot))\;, \quad i,j =1, \dots, d\;.
\end{align*}

Under assumptions (1-3) above, the SDE (28) has a unique weak solution $X$ and it holds that $x_\varepsilon \to_{\lambda_\varepsilon} X$.

%%%%%%%%%%%%%%%%%%%%%%%%%%%%%%%%%%%%%%%%%%%%%%%%%%%%%%%%%%%%%%%%%%%%%%%%%%%%%%%%%%%%%%%%%%%%%%%%%%%%%%%%%%%%%%%%%%%%%%%%%%%%%%%%%%%%%%%%%%%%%%%%%%%%%%%%%%%%%%%

\def\cprime{$'$}

\end{document}